\documentclass[a4paper,11pt,leqno]{article}
\usepackage{amsmath,amsfonts,amsthm,amssymb}
\usepackage[left=3cm, right=3cm, top=3cm, bottom=3.5cm]{geometry}
\usepackage{graphicx}
\usepackage{tikz}
\numberwithin{equation}{section}

\theoremstyle{plain}
\newtheorem{theorem}{Theorem}[section]

\newtheorem{lemma}[theorem]{Lemma}
\newtheorem{proposition}[theorem]{Proposition}
\newtheorem{corollary}[theorem]{Corollary}

\theoremstyle{remark}
\newtheorem{remark}{Remark}

\theoremstyle{definition}

\newcommand{\R}{\mathbb{R}}

\begin{document}

\title{Existence of periodic orbits near heteroclinic connections}
\author{{Giorgio
    Fusco,\footnote{Universit\`a dell'Aquila, ; e-mail:
      {\texttt{fusco@univaq.it}}}}\ \
Giovanni F. Gronchi,\footnote{Dipartimento di Matematica, Universit\`a
  degli Studi di Pisa, Largo B. Pontecorvo 5, Pisa, Italy; e-mail:
  {\texttt{giovanni.federico.gronchi@unipi.it}}}\ \
Matteo Novaga\footnote{Dipartimento di Matematica, Universit\`a degli
  Studi di Pisa, Largo B. Pontecorvo 5, Pisa, Italy; e-mail:
  {\texttt{matteo.novaga@unipi.it}}} }
\date{}
\maketitle

\begin{abstract}
We consider a potential $W:\R^m\rightarrow\R$ with two different
global minima $a_-, a_+$ and, under a symmetry assumption, we use a
variational approach to show that the Hamiltonian system
\begin{equation}\ddot{u}=W_u(u),
\label{*}
\end{equation}
has a family of $T$-periodic solutions $u^T$ which, along
a sequence $T_j\rightarrow+\infty$, converges locally to a heteroclinic
solution that connects $a_-$ to $a_+$. We then focus on the elliptic
system
\begin{equation}
\Delta u=W_u(u),\;\; u:\R^2\rightarrow\R^m,
\label{**}
\end{equation}
that we interpret as an infinite dimensional analogous of \eqref{*},
where $x$ plays the role of time and $W$ is replaced by the action
functional $J_\R(u)=\int_\R(\frac{1}{2}\vert u_y\vert^2+W(u))dy$. We
assume that $J_\R$ has two different global minimizers $\bar{u}_-,
\bar{u}_+:\R\rightarrow\R^m$ in the set of maps that connect $a_-$ to
$a_+$. We work in a symmetric context and prove, via a minimization
procedure, that \eqref{**} has a family of solutions
$u^L:\R^2\rightarrow\R^m$, which is $L$-periodic in $x$,
converges to $a_\pm$ as $y\rightarrow\pm\infty$ and, along a sequence
$L_j\rightarrow+\infty$, converges locally to a heteroclinic solution
that connects $\bar{u}_-$ to $\bar{u}_+$.
\end{abstract}

\tableofcontents

\section{Introduction}
The dynamics of the Newton equation
\begin{equation}
\ddot{u}=W^\prime(u),\qquad \;W(u)=\frac{1}{4}(1-u^2)^2,
\label{equation}
\end{equation}
includes a heteroclinic solution $u^H:\R\rightarrow\R$ that connects
$-1$ to $1$:
\[
 \lim_{t\rightarrow\pm\infty}u^H(t)=\pm 1,
\]
and a family of $T$-periodic solutions $u^T$ that, along a sequence
$T_j\rightarrow+\infty$, converges to $u^H$
\[
\lim_{T\rightarrow+\infty}u^T(t)=u^H(t),
\]
uniformly in compact intervals.

Each map $t\rightarrow u^T(t)$
 satisfies
\[
u^T\left(\frac{T}{4}-t\right)=u^T\left(\frac{T}{4}+t\right),\;\; t\in\R,
\]
and therefore oscillates twice for period on the same trajectory with
extremes at $u^T(\pm\frac{T}{4})$ where the speed
$\dot{u}^T(\pm\frac{T}{4})$ vanishes and for this reason is called a
\emph{brake orbit}. There is a large literature on brake orbits
\cite{Seifert}, \cite{R}, \cite{benci}, \cite{ZL}.

We can ask whether a similar picture holds true in the vector case
where $W:\R^m\rightarrow\R$, $m>1$ satisfies
\begin{equation}
0=W(a_\pm)<W(u),\;\;u\neq a_\pm,
\label{W}
\end{equation}
for some $a_-\neq a_+\in\R^m$, or even in the infinite dimensional
case where the potential $W$ is replaced by a functional
$J:\mathcal{H}\rightarrow\R$, where $\mathcal{H}$ is a suitable
function space, with two distinct global minima
$\bar{u}_\pm\in\mathcal{H}$ that correspond to the zeros $a_\pm$ of
$W$ in the finite dimensional case.

If we assume that $W$ is of class $C^2$ and that $a_\pm$ are non
degenerate in the sense that the Hessian matrix $W_{uu}(a_\pm)$ is
positive definite, the existence of a family of $T$-periodic brake
maps that, as $T\rightarrow+\infty$, converges to a heteroclinic
connection between $a_-$ and $a_+$ can be established by direct
minimization of the action functional
\[
 J_{(t_1^u,t_2^u)}(u)=\int_{t_1^u}^{t_2^u}\Big(\frac{1}{2}\vert\dot{u}\vert^2+W(u)\Big)ds,\qquad
 -\infty<t_1^u<t_2^u<+\infty,
\]
on a suitable set of admissible maps $u\in
H^1((t_1^u,t_2^u);\R^m)$. Indeed the non degeneracy of $a_\pm$ implies
that, for small $\delta>0$, the boundary of the set $\{u\in\R^m:
W(u)>\delta\}$ is partitioned into two compact connected subsets
$\Gamma_-$ and $\Gamma_+$ that satisfy the condition
\begin{equation}
 W_u(u)\neq 0,\;\;u\in \Gamma_\pm.
\label{condition}
\end{equation}
Then Theorem 5.5 in \cite{AS} or Corollary 1.5 in \cite{fgn} yields
the existence of a brake orbit $u^\delta$ that oscillates between
$\Gamma_-$ and $\Gamma_+$ and whose period $T_\delta$ diverges to
$+\infty$ as $\delta\rightarrow 0^+$. Even though the condition
(\ref{condition}) can be relaxed by allowing $\Gamma_\pm$ to contain
hyperbolic critical points of $W$ \cite{fgn}, the extension of this
approach to the infinite dimensional setting requires new ideas to
overcome the difficulties related to the formulation of a condition
analogous to (\ref{condition}) and to the non compactness of the
boundary of the sets $\{u\in\mathcal{H}:
J(u)-J(\bar{u}_\pm)>\delta\}$.  To avoid these pathologies the idea is
to minimize on a set of $T$-periodic maps. But we can not expect that
$u^\delta$ is a minimizer in the class of maps of period
$T=T_\delta$. Indeed, returning to the case $m=1$, we note that, as a
solution of (\ref{equation}), $u^{_T}$ is a critical point of the
action functional
\begin{equation}
J_{(0,T)}(u):=\int_0^T\Big(\frac{1}{2}\vert\dot{u}\vert^2+W(u)\Big) dt,
\label{functional}
\end{equation}
in the set of $H^1$ $T$-periodic maps but is not a minimizer. In fact
it is well known \cite{cp}, \cite{fh}, \cite{bnn} that, in the
dynamics of the scalar parabolic equation
\[
u_\tau=u_{tt}-W^\prime(u),\;\;u(t+T)=u(t),
\]
nearest layers attract each other and therefore, for large $T$,
$u^{_T}$ has Morse index $1$ in the context of periodic perturbations.

To mode out this instability we work in a symmetric context. We assume
that $W$ is invariant under a reflection $\gamma:\R^m\rightarrow\R^m$,
that is,
\begin{equation}
W(\gamma u)=W(u),\;\;u\in\R^m.
\label{Winv}
\end{equation}
In the finite dimensional case we assume that $\gamma$ exchanges $a_-$
with $a_+$:
\begin{equation}
a_\pm=\gamma\, a_\mp,
\label{gammaa}
\end{equation}
and we restrict ourselves to equivariant maps:
\[
u(-t)=\gamma u(t),\;\;t\in\R.
\]
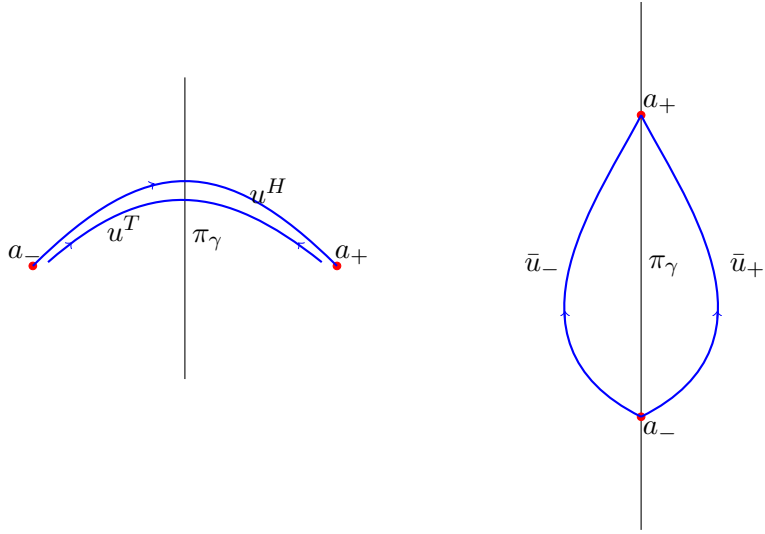
\begin{figure}
  \begin{center}
\begin{tikzpicture}
{\draw (6,-1.5)--(6,2.5);
\node at (6.3,.35){$\pi_\gamma$};
\draw[red, fill] (4,0) circle [radius=0.05];
\draw[red, fill] (8,0) circle [radius=0.05];
\draw [blue, thick](4,0).. controls (5.5,1.5) and(6.5,1.5)..(8,0);
\draw [blue, thick](4.2,+.05).. controls (5.5,1.2) and(6.5,1.1)..(7.8,+.05);
 \node at (3.9,.15){$a_-$};
\node at (8.2,.15){$a_+$};
\node at (7.1,1){$u^H$};
\node at (5.2,.5){$u^T$};
\draw [blue] [->](4.5,.3)--(4.51,.31);
\draw [blue] [->](7.5,.3)--(7.49,.31);
\draw [blue] [->](5.55,1.07)--(5.6,1.095);}

{\draw (12,-3.5)--(12,3.5);
\node at (12.3,0){$\pi_\gamma$};
\draw[red, fill] (12,-2) circle [radius=0.05];
\draw[red, fill] (12,+2) circle [radius=0.05];
\draw [blue, thick](12,-2).. controls (10,-1) and(11.5,1)..(12,+2);
\draw [blue, thick](12,-2).. controls (14,-1) and(12.5,1)..(12,+2);
\node at (12.25,-2.19){$a_-$};
\node at (12.25,2.15){$a_+$};
\node at (10.7,0){$\bar{u}_-$};
\node at (13.4,0){$\bar{u}_+$};
\draw [blue] [->](11,-.7)--(11,-.6);
\draw [blue] [->](13,-.7)--(13,-.6);}
\end{tikzpicture}
\end{center}
\caption{The symmetry of $W$: finite dimension (left); infinite dimension (right)}
\label{figsym}
\end{figure}
We show that, under these restrictions and minimal assumptions on $W$,
the existence of periodic solutions to
\begin{equation}
\ddot{u}=W_u(u),\qquad W_u(u)=\left(\frac{\partial W}{\partial
  u_1}(u),\ldots,\frac{\partial W}{\partial u_m}(u)\right)^\top,
\label{system}
\end{equation}
can be established by minimizing $J_{(0,T)}$ on a suitable set of
$T$-periodic maps.

In the infinite dimensional case our choice for the functional that
replaces $W$ is the action functional
\[
J_\R(u)=\int_\R \Big(\frac{1}{2}\vert
u^\prime\vert^2+W(u)\Big)ds,\;\;u\in\bar{\mathrm{u}}+H^1(\R;\R^m),
\]
where $W$ satisfies (\ref{W}) and $\bar{\mathrm{u}}$ is a smooth map
such that $\lim_{s\rightarrow\pm\infty}\bar{\mathrm{u}}(s)=a_\pm$ with
exponential convergence.  We assume that (\ref{Winv}) holds with
$\gamma$ a reflection that, in analogy with the finite dimensional
case, satisfies
\begin{equation}
\bar{u}_\pm(s)=\gamma\bar{u}_\mp(s),\;\;s\in\R,
\label{barupm}
\end{equation}
with $\bar{u}_-$ and $\bar{u}_+$ distinct global minimizers of $J_\R$
on $\bar{\mathrm{u}}+H^1(\R;\R^m)$. The maps $\bar{u}_-$ and
$\bar{u}_+$ represent two distinct orbits that connect $a_-$ to $a_+$:
\begin{equation}  
  \lim_{s\rightarrow\pm\infty}\bar{u}_\pm(s)=a_\pm.
  \label{het-limit}
\end{equation}
We assume that $\bar{u}_-$ and $\bar{u}_+$ are unique modulo
translation.  Note that (\ref{het-limit}) and (\ref{barupm}) imply
that $a_\pm=\gamma a_\pm$, that is $a_\pm$ belong to the plane
$\pi_\gamma$ fixed by $\gamma$, see Figure \ref{figsym}.  We restrict ourselves to symmetric maps and
replace the dynamical equation (\ref{system}) with
\[
\ddot{u}=\nabla_{L^2}J_\R(u)=-u^{\prime\prime}+W_u(u).
\]
This is actually an elliptic system which, after setting $x=t$ and
$y=s$ takes the form
\begin{equation}
 u_{xx}+u_{yy}=\Delta u=W_u(u).
\label{elliptic}
\end{equation}
We prove that for all $L\geq L_0$, for some $L_0>0$, there is a classical
solution $u^L:\R^2\rightarrow\R^m$ of (\ref{elliptic}) which is equivariant:
\[u^L(-x,y)=\gamma u^L(x,y),\] $L$-periodic in
$x\in\R$ and such that, along a subsequence $L_j\rightarrow+\infty$, converges locally to a
heteroclinic solution that connects $\bar{u}_-$
and $\bar{u}_+$. That is, to a map
$u^H:\R^2\rightarrow\R^m$ that satisfies \eqref{elliptic} and

\begin{equation}
  \begin{split}
 &\lim_{y\rightarrow\pm\infty} u^H(x,y)=a_\pm,\\
    &\lim_{x\rightarrow\pm\infty} u^H(x,y)=\bar{u}_\pm(y).
  \end{split}
  \label{properties1}
\end{equation}
We remark that, in the proof of this, there is an extra difficulty which is not present in the finite dimensional case:
$\bar{u}_-$ and $\bar{u}_+$ are not isolated but any translate $\bar{u}_-(\cdot-r)$ or $\bar{u}_+(\cdot-r)$, $r\in\R$, is again a global minimizer of $J_\R$. Therefore for each $x$ there is a $\bar{u}\in\{\bar{u}_-,\bar{u}_+\}$
and a translation $h(x)$ that determines the point $\bar{u}(\cdot-h(x))$ in the manifolds generated by $\bar{u}_-$ and $\bar{u}_+$ which is the closest to the \emph{fiber} $u^L(x,\cdot)$ of $u^L$. The map $h$ depends on $L$ and to prove convergence to a heteroclinic solution one needs to control $h$ and show that can be bounded by a quantity that does not depend on $L$ and that, for $L_j\rightarrow+\infty$, converges to a limit map $h^\infty:\R\rightarrow\R$ with a definite limit for $x\rightarrow\pm\infty$.
\medskip


The paper is organized as follows. After stating our main results,
that is Theorem \ref{periodic} in Section \ref{finite} and Theorem
\ref{periodic1} in Section \ref{infinite}, we prove Theorem
\ref{periodic} and Theorem \ref{periodic1} in Sections \ref{finite1}
and \ref{infinite1} respectively. The approach used in Section
\ref{infinite1} is inspired by \cite{f}. We include an Appendix where
we present an elementary proof of a property of the functional $J_\R$.

\subsection{The finite dimensional case}
\label{finite}
We assume that $W:\R^m\rightarrow\R$ is a continuous function that
satisfies (\ref{W}), (\ref{Winv}) and (\ref{gammaa}). We also assume
that there is a non-negative function $\sigma:[0,+\infty)\rightarrow\R$
  such that $\int_0^{+\infty}\sigma(r)dr=+\infty$ and\footnote{This
    condition was first introduced in \cite{monteil}}
\begin{equation}
  {\sqrt{W(z)}\geq\sigma(\vert z\vert),\;\; z\in\R^m.}
  \label{sigma}
\end{equation}

\begin{remark}
The assumptions on $W$ imply (see for example \cite{monteil},
\cite{ZS} and \cite{fgn}) the existence of a Lipschitz continuous map
$u^H:\R\rightarrow\R^m$ that satisfies
\begin{equation}
  \begin{split}
    &\lim_{t\rightarrow\pm\infty}u(t)=a_\pm,\\
    &\frac{1}{2}\vert\dot{u}\vert^2-W(u)=0,\\
    &u(-t)=\gamma u(t),\;\;t\in\R.
  \end{split}
  \label{properties0}
\end{equation}
We refer to a map with these properties as a heteroclinic connection
between $a_-$ and $a_+$.
\label{remark1}
\end{remark}

Define
\begin{equation}
  \mathcal{A}^T := \biggl\{u\in H^1_T(\R;\R^m),\;
  u\Bigl(\frac{T}{4}+t\Bigr) = u\Bigl(\frac{T}{4}-t\Bigr),\;
  u(-t)=\gamma u(t),\;\;t\in\R\biggr\},
  \label{admissible}
\end{equation}
and observe that there exists $\tilde{u}\in\mathcal{A}^T$ and a
constant $C_0>0$ independent of $T>4$ such that
\begin{equation}
  J_{(0,T)}(\tilde{u})\leq C_0.
  \label{energy-bound}
\end{equation}
Indeed the map $\tilde{u}$ can be defined by
\[
  \begin{split}
    &\tilde{u}(t)=\frac{1}{2}(a_++a_-+t(a_+-a_-)),\;\;t\in[-1,1],\\
    &\tilde{u}(t)=a_+,\;\;t\in[1,\frac{T}{2}-1].
  \end{split}
\]
Since we are interested in periodic orbits near $u^H$ we restrict our
search to orbits lying in a large ball.  Fix $M$ as the solution of the
equation
\begin{equation}
C_0=\sqrt{2}\int_{2(|a_+|\lor|a_-|)}^M\sigma(s)ds.
\label{Mdef}
\end{equation}
We determine $T$-periodic maps near heteroclinic solutions by
minimizing the action functional (\ref{functional}) on the set
$\mathcal{A}^T\cap\{\|u\|_{L^\infty}\leq 2M\}$.

\begin{theorem}
Assume that $W:\R^m\rightarrow\R$ is a continuous function that
satisfies \eqref{W}, \eqref{Winv}, \eqref{gammaa} and \eqref{sigma}.
Then, there exists $T_0$ such that for each $T\geq T_0$ there exists a
$T$-periodic minimizer $u^T$ of the functional \eqref{functional} in
$\mathcal{A}^T\cap\{\|u\|_{L^\infty}\leq 2M\}$, which is
Lipschitz continuous and satisfies
\begin{enumerate}
\item  $J_{(0,T)}(u^T)\leq C_0,\quad\quad\|u^T\|_{L^\infty}\leq M$,
\item $u^T(-t)=\gamma u^T(t)$,
\item $\frac{1}{2}\vert\dot{u}^T\vert^2-W(u^T)=-W(u^T(\pm\frac{T}{4})), a.e.$
\end{enumerate}
For each $0<q\leq q_0$, for some $q_0>0$, there is a $\tau_q>0$
such that for each $T>4\tau_q$
\begin{equation}
\begin{split}
  &\vert u^T(t)-a_+\vert< q,\;\;t\in\left[\tau_q,\displaystyle\frac{T}{2}-\tau_q\right],\;\;q\in(0,q_0],\cr
\end{split}
\label{*2}
\end{equation}
and therefore
\[
\lim_{T\rightarrow+\infty}u^T\Bigl(\pm\frac{T}{4}\Bigr)=a_\pm.
\]
Moreover, there is a sequence $T_j\rightarrow+\infty$ and a
heteroclinic connection between $a_-$ and $a_+$
$u^H:\R\rightarrow\R^m$ such that
\[
\lim_{j\rightarrow+\infty}u^{T_j}(t)=u^H(t),
\]
uniformly in compacts.

If $W$ is of class $C^1$, then $u^T$ is a classical $T$-periodic
solution of \eqref{system}.
\label{periodic}
\end{theorem}

Note that, if $a_\pm$ is nondegenerate in the sense that the Hessian
matrix $W_{uu}(a_\pm)$ is positive definite or, more generally, if
\[
W_u(u)\cdot(u-a_\pm)\geq\mu\vert u-a_\pm\vert^2,\;\;\text{ for }\;\vert u-a_\pm\vert\leq r_0,
\]
for some $\mu>0$, $r_0>0$, then \eqref{*2} can be strengthened to 
\[
\vert u^T(t)-a_+\vert\leq Ce^{-ct},\;\;t\in\Bigl[0,\frac{T}{4}\Bigr],
\]
where $c, C$ are positive constants independent of $T$.
This follows by 
\[
\frac{d^2}{d t^2}\vert u^T(t)-a_+\vert^2\geq 2\ddot{u}^T\cdot(u^T-a_+)=2W_u(u^T)\cdot(u^T-a_+)\geq2\mu\vert u^T-a_+\vert^2,
\]
and a comparison argument.

\begin{remark}
Depending on the behavior of $W$ in a neighborhood of $a_\pm$ it may
happen that the map $u^H$ connects $a_-$ and $a_+$ in a finite time,
that is, $\exists\:\tau_0<+\infty:
u^H((-\tau_0,\tau_0))\cap\{a_-,a_+\}=\emptyset$,
$u^H(\pm\tau_0)=a_\pm$. We do not exclude this case. A sufficient
condition for $\tau_0=+\infty$, is
\[
W(u)\leq c\vert u-a_\pm\vert^2,
\]
for $u$ in a neighborhood of $a_\pm$.

Note that, if $\tau_0<+\infty$, one can immediately construct a
$T$-periodic map $u^T$ ($T=4\tau_0$) that satisfies
(\ref{properties0}), by setting
\[
u^T\Bigl(\frac{T}{4}+t\Bigr)=u^H\Bigl(\frac{T}{4}-t\Bigr),\;\;\text{
  for }\;t\in\Bigl(0,\frac{T}{2}\Bigr).
\]
\label{remark2}
\end{remark}

\subsection{The infinite dimensional case}
\label{infinite}
We assume that $W:\R^m\rightarrow\R$ is of class $C^3$, that
(\ref{W}), (\ref{Winv}) and (\ref{barupm}) hold with $\bar{u}_\pm$ as
before. Moreover we assume
\begin{description}
\item[$\mathbf{h}_1$]
  $\liminf_{\vert u\vert\rightarrow+\infty}W(u)>0$ and there is $M>0$ such that
  \begin{equation}
    W(su)\geq W(u),\;\;\text{ for}\;\vert u\vert = M,\;s\geq 1.
    \label{W-mon}
  \end{equation}
  
\item[$\mathbf{h}_2$] $a_\pm$ are non degenerate in the sense that the
  Hessian matrix $W_{uu}(a_\pm)$ is definite positive.
  
  For each $r\in\R$ $\bar{u}(\cdot-r)$,
  $\bar{u}\in\{\bar{u}_-,\bar{u}_+\}$, is a solution of
  (\ref{system}).  Therefore differentiating (\ref{system}) with
  respect to $r$ yields
  $\bar{u}^{\prime\prime\prime}=W_{uu}(\bar{u})\bar{u}^\prime$ that
  shows that $0$ is an eigenvalue of the operator
  $T:H^2(\R;\R^m)\rightarrow L^2(\R;\R^m)$ defined by
  \[
  T v=-v^{\prime\prime}+W_{uu}(\bar{u})v,\;\;\bar{u}=\bar{u}_\pm,
  \]
  and $\bar{u}^\prime$ is a corresponding eigenvector.
  
  We also assume
\item[$\mathbf{h}_3$] The maps $\bar{u}_\pm$ are non degenerate in the
  sense that $0$ is a simple eigenvalue of $T$.
  
\end{description}


The above assumptions ensure the existence of a heteroclinic
connection between $\bar{u}_-$ and $\bar{u}_+$.  This was proved by
Schatzman in \cite{scha} without restricting to equivariant maps (see
also \cite{f} and \cite{monteil1}). The first existence result for a
heteroclinic that connects $\bar{u}_-$ to $\bar{u}_+$ was given in
\cite{abg} under the assumption that $W$ is symmetric with respect to
the reflection that exchanges $a_\pm$ with $a_\mp$ but without
requiring \eqref{barupm}.

\begin{remark}
It is well known that the non-degeneracy of $a_\pm$ implies
\begin{equation}
  \begin{split}
    &\vert\bar{u}(y)-a_+\vert\leq
    Ke^{-ky},\;\;y>0,\quad\vert\bar{u}(y)-a_-\vert\leq
    Ke^{ky},\;\;y<0,\\
    &\vert\bar{u}^\prime(y)\vert,
    \vert\bar{u}^{\prime\prime}(y)\vert\leq Ke^{-k\vert
      y\vert},\;\;y\in\R,
  \end{split}
  \label{bar-exp}
\end{equation}
for some constants $k>0,K>0$.
\label{remark3}
\end{remark}

Under the above assumptions we prove the following:
\begin{theorem}
There is $L_0>0$ and positive constants $k,K$, $k^\prime,K^\prime$ such that for each $L\geq L_0$ there exists a classical
solution $u^L:\R^2\rightarrow\R^m$ of \eqref{elliptic}, with the
following properties:
\begin{enumerate}
\item 
  $\vert u^L(x,y)-a_-\vert\leq Ke^{ky},\;\;x\in\R,\;y\leq 0$,
  \vskip.1cm
  $\vert u^L(x,y)-a_+\vert\leq Ke^{-ky},\;\;x\in\R,\;y\geq 0$.
  
\item $u^L$ is $L$-periodic in $x\in\R$:
  $u^L(x+L,y)=u^L(x,y),\;\;(x,y)\in\R^2$.
\item $u^L$ is a brake orbit: $u^L(\frac{L}{4}+x,y)=u^L(\frac{L}{4}-x,y),$
\item $u^L$ is equivariant $u^L(-x,y)=\gamma u^L(x,y)$
\item $u^L$ satisfies the identities:
  \vskip.1cm
  $\frac{1}{2}\|u_x^L(x,\cdot)\|_{L^2(\R;\R^m)}^2-J_\R(u^L(x,\cdot))=-J_\R(u^L(\frac{L}{4},\cdot))$,
  \vskip.1cm
  $\hskip.5cm\langle u_x^L(x,\cdot),u_y^L(x,\cdot)\rangle_{L^2(\R;\R^m)}=0,\;x\in\R.$
  \vskip.25cm
\item $u^L$ minimizes
  \[
  \mathcal{J}(u)=\int_{(0,L)\times\R}\Big(\frac{1}{2}\vert\nabla u\vert^2+W(u)\Big)dxdy
  \]
  on the set of the $H_{\mathrm{loc}}^1(\R^2;\R^m)$ maps that satisfy $(ii)-(iv)$ and
  $
  \lim_{y\rightarrow\pm\infty}u(x,y)=a_\pm.
  $

\item 
  $
  \min_{r\in\R}\|u^L(x,\cdot)-\bar{u}_+(\cdot-r)\|_{L^2(\R;\R^m)}\leq K^\prime e^{-k^\prime x},\;\;x\in[0,\frac{L}{4}].
  $
  
  In particular, as $L\rightarrow+\infty$, $u^L(\frac{L}{4},\cdot)$
  converges to the manifold of the translates of $\bar{u}_+$.
\end{enumerate}
Moreover, there exist $\eta\in\R$, a sequence $L_j\rightarrow+\infty$ and a
heteroclinic solution $u^H:\R^2\rightarrow\R^m$ connecting 
$\bar{u}_-$ to $\bar{u}_+$ that satisfy
\[
\lim_{j\rightarrow+\infty}u^{L_j}(x,y-\eta)=u^H(x,y),\;\;(x,y)\in\R^2,
\]
uniformly in $C^2$ in any strip of the form $(-l,l)\times\R$,
for $l>0$.
\label{periodic1}
\end{theorem}
Note that a by product of this theorem is a new proof of the existence
of a heteroclinic solution $u^H$ in the class of equivariant maps.

\section{The proof of Theorem \ref{periodic}}
\label{finite1}

From (\ref{energy-bound}) we can restrict ourselves to consider maps in the
subset
\begin{equation}
  \mathcal{A}^T_{C_0,M} =
  \{u\in\mathcal{A}^T\cap\{\|u\|_{L^\infty}\leq 2M\}: J_{(0,T)}(u)\leq C_0\},
\label{subset}
\end{equation}
where $M$ is given by \eqref{Mdef}.

\noindent
Step 1.
$\quad u\in\mathcal{A}^T_{C_0,M} \quad\Rightarrow\quad\|u\|_{L^\infty}\leq M.$

\noindent
Define
\begin{equation}
 W_m(s)=\min_{\vert u-a_\pm\vert\geq s,\,\vert u\vert<2M}W(u),\\
 \label{wm-wM}
\end{equation}
Since $u\in\mathcal{A}^T$ implies $u(0)=\gamma u(0)$ we have
\[
\vert u(0)-a_\pm\vert\geq\frac{1}{2}\vert a_+-a_-\vert.
\]
Therefore, given $p\in(0,\frac{1}{2}\vert a_+-a_-\vert)$, for
$u\in\mathcal{A}^T_{C_0,M}$, there are $t_p\in\bigl(0,\frac{C_0}{W_m(p)}\bigr)$
and $a\in\{a_-,a_+\}$ such that, for $T>4t_p$, it results
\begin{equation}
  \begin{split}
    &\vert u(t)-a_\pm\vert>p,\;\;\text{ for }\;t\in[0,t_p),\\
      &\vert u(t_p)-a\vert=p.
  \end{split}
  \label{tp}
\end{equation}
Note, in passing, that since $u\in\mathcal{A}^T$ implies
$u(\frac{T}{4}-t)=u(\frac{T}{4}+t)$ we also have
\begin{equation}
  \Bigl\vert u\Bigl(\frac{T}{2}-t_p\Bigr)-a\Bigr\vert=p.
  \label{tp2}
\end{equation}

Let $\bar{t}$ be such that $\vert u(\bar{t})\vert=\|u\|_{L^\infty}$,
then we have
\[
\begin{split}
  &\sqrt{2}\int_{2(|a_+|\lor|a_-|)}^M\sigma(s)ds=C_0\geq
  J_{(t_p,\bar{t})}(u)\\
  &\geq\int_{t_p}^{\bar{t}}\sqrt{2 W(u(t))}\vert\dot{u}(t)\vert
  dt\geq\sqrt{2}\int_{\vert a\vert+p}^{\|u\|_{L^\infty}}\sigma(s)ds
\end{split}
\]
that proves the claim.

It follows that the constraint $\|u\|_{L^\infty}\leq 2M$ imposed in
the definition of the admissible set is inactive for any
$u\in\mathcal{A}_{C_0,M}^T$.

Next we prove a key lemma which is a refinement of Lemma 3.4 in
\cite{af} based on an idea from \cite{sourdis}.
\begin{lemma}
Assume that $u\in H^1((\alpha,\beta);\R^m)$, $(\alpha,\beta)\subset\R$
a bounded interval, satisfies
\[
\begin{split}
& J_{(\alpha,\beta)}(u)\leq C^\prime,\\
& \|u\|_{L^\infty}\leq M^\prime
\end{split}
\]
for some $C', M'>0$. Let $q_0=\frac{1}{2}\vert a_+-a_-\vert$. Given
$q\in(0,q_0]$, there is $q^\prime(q)\in(0,q)$ such that, if
\[
\begin{split}
  & \vert u(t_i)-a_+\vert\leq q^\prime(q),\;i=1,2\\
  & \vert u(t^*)-a_+\vert\geq q,\;\text{ for some }\;t^*\in(t_1,t_2),
\end{split}
\]
for some $\alpha\leq t_1<t_2\leq\beta$, then there exists $v$ which
coincides with $u$ outside $(t_1,t_2)$ and is such that
\[
\begin{split}
  & \vert v(t)-a_+\vert< q,\;\text{ for }\;t\in[t_1,t_2],\\
  & J_{(t_1,t_2)}(v)<J_{(t_1,t_2)}(u).
\end{split}
\]
\label{lemma1}
\end{lemma}
\begin{proof}
For $t,t^\prime\in\R$, we have
\begin{equation}
  \vert
  u(t)-u(t^\prime)\vert\leq\Bigl\vert\int_t^{t^\prime}\vert\dot{u}\vert
  ds\Bigr\vert \leq\vert
  t-t^\prime\vert^\frac{1}{2}\Bigl(\int_t^{t^\prime}\vert\dot{u}\vert^2
  ds\Bigr)^\frac{1}{2} \leq\sqrt{C_0}\vert
  t-t^\prime\vert^\frac{1}{2}.
  \label{holder}
\end{equation}
Define the intervals
$(\tilde{\tau}_1,\tilde{\tau}_2)\subset(\tau_1,\tau_2)$ by setting
\[
\begin{split}
  &\tilde{\tau}_1=\max\{t>t_1: \vert u(s)-a_+\vert\leq q,\;\text{ for }\;s\leq t\},\\
  &\tau_1=\max\{t<\tilde{\tau}_1: \vert u(t)-a_+\vert\leq q^\prime\},\\
  &\\
  &\tilde{\tau}_2=\min\{t<t_2: \vert u(s)-a_+\vert\leq q,\;\text{ for }\;s\geq t\},\\
  &\tau_2=\min\{t>\tilde{\tau}_2: \vert u(t)-a_+\vert\leq q^\prime\}.
\end{split}
\]

From (\ref{holder}) we have
\[
q-q^\prime=\vert u(\tilde{\tau}_1)-a_+\vert-\vert u(\tau_1)-a_+\vert
\leq\vert u(\tilde{\tau}_1)-
u(\tau_1)\vert\leq\sqrt{C_0}\vert\tau_1-\tilde{\tau}_1\vert^\frac{1}{2},
\]
and therefore
\[
  \tilde{\tau}_1-\tau_1\geq\frac{1}{C_0}(q-q^\prime)^2,
\]
and similarly for $\tau_2-\tilde{\tau}_2$.  Next we set
$\delta_{q,q^\prime}:=\frac{1}{C_0}(q-q^\prime)^2$ and, see Figure
\ref{fig1}, define $v$:
\[
  v=\left\{\begin{array}{l}
  u,\;\;\text{ for }\;t\not\in(\tau_1,\tau_2),\\
  a_+,\;\;\text{ for }\;t\in(\tau_1+\delta_{q,q^\prime},\tau_2-\delta_{q,q^\prime}),\\
  u(\tau_1)-(u(\tau_1)-a_+)\frac{t-\tau_1}{\delta_{q,q^\prime}},
  \;\;\text{ for }\;t\in(\tau_1,\tau_1+\delta_{q,q^\prime}),\\
  u(\tau_2)-(u(\tau_2)-a_+)\frac{\tau_2-t}{\delta_{q,q^\prime}},
  \;\;\text{ for }\;t\in(\tau_2-\delta_{q,q^\prime},\tau_2).
  \end{array}\right.
\]

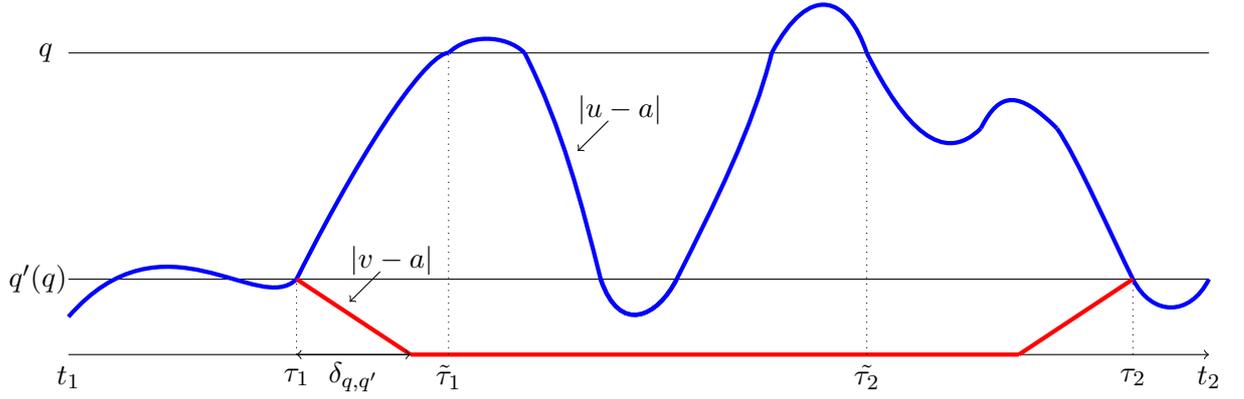
\begin{figure}\label{fig1}
\begin{center}
  \begin{tikzpicture}
\draw [->](0,0)--(15,0);
\draw (0,1)--(15,1);
\draw (0,4)--(15,4);
\draw [blue, ultra thick](0,.5).. controls (1.25,1.95) and(2.5,.5)..(3,1);
\draw [blue, ultra thick](3,1).. controls (4,3) and(4.75,4)..(5,4);
\draw [blue, ultra thick](5,4).. controls (5.25,4.25) and(5.75,4.25)..(6,4);
\draw [blue, ultra thick](6,4).. controls (6.5,3) and(6.75,2)..(7,1);
\draw [blue, ultra thick](7,1).. controls (7.25,.25) and(7.75,.5)..(8,1);
\draw [blue, ultra thick](7,1).. controls (7.25,.25) and(7.75,.5)..(8,1);
\draw [blue, ultra thick](8,1).. controls (8.5,2) and(9,3)..(9.25,4);
\draw [blue, ultra thick](9.25,4).. controls (9.75,4.95) and(10.25,4.75)..(10.5,4);
\draw [blue, ultra thick](10.5,4).. controls (11,3) and(11.5,2.5)..(12,3);
\draw [blue, ultra thick](12,3).. controls (12.25,3.5) and(12.5,3.5)..(13,3);
\draw [blue, ultra thick](13,3).. controls (13.25,2.65) and(13.75,1.5)..(14,1);
\draw [blue, ultra thick](14,1).. controls (14.25,.5) and(14.75,.5)..(15,1);
\draw [red, ultra thick](3,1)--(4.5,0);
\draw [red, ultra thick](4.5,0)--(12.5,0);
\draw [red, ultra thick](12.5,0)--(14,1);
\draw [dotted](3,0)--(3,1);
\draw [dotted](5,0)--(5,4);
\draw [dotted](14,0)--(14,1);
\draw [dotted](10.5,0)--(10.5,4);
\node at (0,-.3){$t_1$};
\node at (3,-.3){$\tau_1$};
\node at (5,-.3){$\tilde{\tau}_1$};
\node at (10.5,-.3){$\tilde{\tau_2}$};
\node at (14,-.3){$\tau_2$};
\node at (15,-.3){$t_2$};
\node at (-.3,4){$q$};
\node at (-.4,1){$q^\prime(q)$};
\draw [<-](6.7,2.7)--(7.1,3.1);
\node at (7.25,3.25){$|u-a|$};
\draw [<-](3.7,.7)--(4.1,1.1);
\node at (4.25,1.25){$|v-a|$};
\draw [<->](3,0)--(4.5,0);
\node at (3.75,-.3){$\delta_{q,q^\prime}$};
\end{tikzpicture}
\end{center}
  \caption{The construction of the map $v$ in Lemma \ref{lemma1}}
\end{figure}

For each $s\in(0,q_0]$ define
\[
  W_M(s)=\max_{\vert u-a_+\vert\leq s}W(u).
\]
We observe that $\vert u(\tau_i)-a_+\vert=q^\prime$, $i=1,2$ and estimate
\[
\begin{split}
& J_{(\tau_1,\tau_2)}(v)=J_{(\tau_1,\tau_1+\delta_{q,q^\prime})}(v)+J_{(\tau_2-\delta_{q,q^\prime},\tau_2)}(v)\leq 2\Bigl(\frac{1}{2}\frac{{q^\prime}^2}{\delta_{q,q^\prime}}+\delta_{q,q^\prime}W_M(q^\prime)\Bigr),\\
& J_{(\tau_1,\tau_2)}(u)\geq J_{(\tau_1,\tilde{\tau}_1)}(u)+J_{(\tilde{\tau}_2,\tau_2)}(u)\\
& \geq
\int_{\tau_1}^{\tilde{\tau}_1}\sqrt{2W(u)}\vert\dot{u}\vert dt+\int_{\tilde{\tau}_2}^{\tau_2}\sqrt{2W(u)}\vert\dot{u}\vert dt\\
&\geq 2\int_{q^\prime}^q\sqrt{2W_m(s)} ds.
\end{split}
\]
where $W_m(s)$ is defined as in (\ref{wm-wM}) with $M^\prime$ instead
of $2M$.

Since $\delta_{q,q^\prime}\leq\frac{q^2}{C_0}$ is a decreasing
function of $q^\prime\in(0,q)$ and $W_M(q')$ is infinitesimal with
$q'$
we can fix a $q^\prime=q^\prime(q)$ so small that
\[
\frac{1}{2}\frac{{q^\prime}^2}{\delta_{q,q^\prime}}+\delta_{q,q^\prime}W_M(q^\prime)<\int_{q^\prime}^q\sqrt{2W_m(s)}
ds.
\]
The proof is complete.
\end{proof}

\noindent
Step 2. From Step 1 and Lemma \ref{lemma1} it follows that, if in
(\ref{tp}) and (\ref{tp2}) we take $p=q^\prime(q)$ and set
$\tau_q=t_{q^\prime(q)}$, then, for $T>4\tau_q$, in the minimization
process we can restrict ourselves to the maps $u\in\mathcal{A}_{C_0,M}^T$ that
satisfy
\[
\begin{split}
  &\vert u(t)-a_+\vert<q,\;t\in\Bigl[\tau_q,\frac{T}{2}-\tau_q\Bigr],\\
  &\vert u(t)-a_-\vert<q,\;t\in\Bigl[\frac{T}{2}+\tau_q,T-\tau_q\Bigr].\\
\end{split}
\]

\noindent
Step 3. The existence of a minimizer $u^T\in\mathcal{A}_{C_0,M}^T$ is
quite standard. From Step 1 and (\ref{holder}) $\mathcal{A}_{C_0,M}^T$
is an equibounded and equicontinuous family of maps. Therefore from
Ascoli-Arzela theorem there exists a minimizing sequence
$\{u_j\}_j\subset\mathcal{A}_{C_0,M}^T$ that converges uniformly to a
map $u^T\in\mathcal{A}_{C_0,M}^T$. This and $J_{(0,T)}(u_j)\leq C_0$
imply that $\{u_j\}_j$ is bounded in $H^1((0,T);\R^m)$ and therefore,
by passing to a subsequence if necessary, that $u_j$ converges
weakly in $H^1$ to $u^T$. From the lower semicontinuity of the norm we
have $\liminf_{j\rightarrow+\infty}\int_0^T\vert
\dot{u}_j\vert^2dt\geq\int_0^T\vert\dot{u}^T\vert^2dt$ while uniform
convergence implies $\lim_{j\rightarrow+\infty}\int_0^T
W(u_j(t))dt\geq\int_0^T W(u^T(t))dt$.

\smallbreak
  \noindent
Step 4. The minimizer $u^T$ is Lipschitz continuous and satisfies
conservation of energy.  Let $t_0<t_1<t_2<t_3$ be numbers such that
$t_3-t_0\leq T$. Given a small number $\xi\in\R$ let
$\phi:\R\rightarrow\R$ be the $T$-periodic piecewise-linear map that
satisfies $\phi(t_0)=t_0$, $\phi(t_1+\xi)=t_1$, $\phi(t_2+\xi)=t_2$,
$\phi(t_3)=t_3$ and let $\psi$ be the inverse of $\phi$. Set
\[
v_\xi(t)=u^T(\phi(t))
\]
and
\[
f(\xi)=J_{(0,T)}(v_\xi)-J_{(0,T)}(u^T).
\]
The minimality of $u^T$ implies that $f^\prime(0)=0$. A simple computation yields
\[
\begin{split}
  &f(\xi)=\int_{(t_0,t_1)\cup(t_2,t_3)}\biggl(\frac{1}{2}
  \Bigl(\frac{1}{\psi^\prime(\tau)}-1\Bigr)\vert\dot{u}^T(\tau)\vert^2+(\psi^\prime(\tau)-1)W\bigl(u^T(\tau)\bigr)\biggl)d\tau\\
  &=\int_{t_0}^{t_1}\biggl(\frac{-\xi}{2(t_1-t_0+\xi)}\vert\dot{u}^T(\tau)\vert^2
  +\frac{\xi}{t_1-t_0}W\bigl(u^T(\tau)\bigr)\biggr)d\tau\\
  &+\int_{t_2}^{t_3}\biggl(\frac{\xi}{2(t_3-t_2-\xi)}\vert\dot{u}^T(\tau)\vert^2
  -\frac{\xi}{t_3-t_2}W\bigl(u^T(\tau)\bigr)\biggr)d\tau,
\end{split}
\]
and we obtain
\[
\begin{split}
  &0=f^\prime(0)\\
  &\Leftrightarrow\frac{1}{t_1-t_0}\int_{t_0}^{t_1}\Big(\frac{1}{2}\vert\dot{u}^T(\tau)\vert^2
  -W\bigl(u^T(\tau)\bigr)\Big)d\tau=\frac{1}{t_3-t_2}\int_{t_2}^{t_3}\Big(\frac{1}{2}\vert\dot{u}^T(\tau)\vert^2
  -W\bigl(u^T(\tau)\bigr)\Big)d\tau.
\end{split}
\]
This shows that there exists $C\in\R$ independent of $t$ such that
\[
\lim_{t^\prime\rightarrow t}\frac{1}{t^\prime-t}\int_{t}^{t^\prime}\Big(\frac{1}{2}\vert\dot{u}^T(\tau)\vert^2
-W\bigl(u^T(\tau)\bigl)\Big)d\tau=C.
\]
Therefore we have
\[
\frac{1}{2}\vert\dot{u}^T(t)\vert^2
-W\bigl(u^T(t)\bigr)=C
\]
for each Lebesgue point $t\in\R$.  From
$u(\frac{T}{4}-t)=u(\frac{T}{4}+t)$ it follows that
$\dot{u}(\frac{T}{4})=0$, which implies $C=W\bigl(u^T(\pm\frac{T}{4})\bigr)$.

\noindent
Step 5.
If $W$ is of class $C^1$, then $u^T$ is a classical solution of (\ref{system}).
Since $u^T$ is a minimizer, if $w:(t_1,t_2)\rightarrow\R^m$ is a smooth map that satisfies $w(t_i)=0$, $i=1,2$ we have
\begin{equation}
  \begin{split}
    & 0=\frac{d}{d\lambda}J_{(t_1,t_2)}(u^T+\lambda w)\vert_{\lambda=0}\\
    &=\int_{t_1}^{t_2}(\dot{u}^T\cdot\dot{w}+W_u(u^T)\cdot w)dt=\int_{t_1}^{t_2}\Bigl(\dot{u}^T-\int_{t_1}^t W_u\bigl(u^T(s)\bigr)ds\Bigr)\cdot\dot{w}dt.
  \end{split}
  \label{VarEq}
\end{equation}
Since this is valid for all $0<t_1<t_2<T$ and
$\dot{w}:(t_1,t_2)\rightarrow\R^m$ is an arbitrary smooth map with
zero average (\ref{VarEq}) implies
\[
\dot{u}^T=\int_{t_1}^t W_u\bigl(u^T(s)\bigr)ds+const.
\]
The continuity of $u^T$ and of $W_u$ implies that the right hand side
of this equation is a map of class $C^1$. It follows that we can
differentiate and obtain
\[
\ddot{u}^T=W_u(u^T),\;\;t\in(0,T).
\]
The proof of Theorem \ref{periodic} is complete.

\section{The proof of Theorem \ref{periodic1}}
\label{infinite1}

In analogy with the finite dimensional case we define
\[
\begin{split}
  &\mathcal{A}^L=\Bigl\{u\in H_{\mathrm{loc}}^1(\R^2;\R^m): u(x+L,\cdot)=u(x,\cdot),
  \lim_{y\rightarrow\pm\infty}u(x,y)=a_\pm,\\
  &\quad\quad u\Bigl(\frac{L}{4}+x,y\Bigr)=u\Bigl(\frac{L}{4}-x,y\Bigr),\;\;
  u(-x,y)=\gamma u(x,y).\Bigr\}
\end{split}
\]
We will show that the solution of (\ref{elliptic}) in Theorem
\ref{periodic1} can be determined as a minimizer of the energy
\begin{equation}
  \mathcal{J}_{(0,L)\times\R}(u)=\int_{(0,L)\times\R}\Big(\frac{1}{2}\vert\nabla
  u\vert^2+W(u)\Big) dxdy
  \label{functional1}
\end{equation}
on $\mathcal{A}^L$.

We can assume
\begin{equation}
  \|u\|_{L^\infty(\R^2;\R^m)}\leq M,
  \label{u-bound1}
\end{equation}
where $M$ is the constant in $\mathbf{h}_1$ and
\begin{equation}
  \mathcal{J}_{(0,L)\times\R}(u)\leq C_0+c_0L,
  \label{energy-bound1}
\end{equation}
where $C_0>0$ is a constant independent of $L>4$ and
\begin{equation}
  c_0=J_\R(\bar{u}_\pm).
  \label{c0def}
\end{equation}
To prove (\ref{u-bound1}) set $u_M=0$ if $u=0$ and $u_M=\min\{\vert
u\vert, M\}{u/\vert u\vert}$ otherwise and note that (\ref{W-mon})
implies
\[
W(u_M)\leq W(u),
\]
while
\[
\vert\nabla u_M\vert\leq\vert\nabla u\vert,\;\;\text{a.e.},
\]
because the mapping $u\rightarrow u_M$ is a projection.
It follows
\[
\begin{split}
  &\mathcal{J}_{(0,L)\times\R}(u)-\mathcal{J}_{(0,L)\times\R}(u_M)\\
  &=\int_{\{\vert u\vert\geq M\}}\Big(W(u)-W(u_M)
  +\frac{1}{2}(\vert\nabla u\vert^2-\vert\nabla u_M\vert^2)\Big)dxdy\geq 0,
\end{split}
\]
that proves the claim.  To prove (\ref{energy-bound1}) we define a map
$\tilde{u}\in\mathcal{A}^L$ that satisfies (\ref{energy-bound1}) by
setting:
\[
  \begin{split}
    &\tilde{u}(x,\cdot)=\frac{1}{2}(\bar{u}_++\bar{u}_-+x(\bar{u}_+-\bar{u}_-)),\;\;x\in[-1,1],\\
    &\tilde{u}(x,\cdot)=\bar{u}_+,\;\;x\in[1,\frac{L}{2}-1].
  \end{split}
\]

\begin{remark}
  From \eqref{energy-bound1} and the minimality of ${\bar u}_\pm$ it
  follows that
  \[
  \int_0^L\int_\R \frac{1}{2}|u_x|^2dxdy \leq C_0 + c_0L -  \int_0^L\int_\R \Bigl(\frac{1}{2}|u_y|^2 + W(u)\Bigr)dxdy \leq C_0.
  \]
\end{remark}

Since $a_\pm$ are non degenerate zeros of $W\geq 0$, there exist
positive constants $\gamma, \Gamma$ and $r_0>0$ such that
\begin{equation}
  \begin{split}
    &W_{uu}(a_\pm+z)\zeta\cdot\zeta\geq\gamma^2\vert\zeta\vert^2,\;\;\zeta\in\R^m,\;\vert z\vert\leq r_0,\\
    &\frac{1}{2}\gamma^2\vert z\vert^2\leq W(a_\pm+z)\leq\frac{1}{2}\Gamma^2\vert z\vert^2,\;\;\vert z\vert\leq r_0.
  \end{split}
  \label{second-derW}
\end{equation}
For a map $v:\R\rightarrow\R^m$ we simply denote the norms
$\|v\|_{L^2(\R;\R^m)}$ and $\|v\|_{H^1(\R;\R^m)}$ with $\|v\|$ and
$\|v\|_1$ respectively.
\vskip.1cm One of the difficulties with the minimization on
$\mathcal{A}^L$ is the fact that $\mathcal{J}_{(0,L)\times\R}$ is
translation invariant on $\mathcal{A}^L$. This corresponds to a loss
of compactness.  We show in the next lemma that we can restrict ourselves to a
subset of $\mathcal{A}^L$ of maps $u$ that, aside from a bounded
interval independent of $u$, remain near to $a_-$ and $a_+$. This
restores compactness.
\begin{lemma}
  There is $d_L>0$ such that in the minimization of the functional
  (\ref{functional1}) on $\mathcal{A}^L$ we can restrict ourselves to
  the subset of maps that satisfy
\begin{equation}
  \begin{split}
    &\vert u(x,y)-a_-\vert<\frac{r_0}{2},\;\text{for}\;x\in\R,\,y<-d_L,\\
    &\vert u(x,y)-a_+\vert<\frac{r_0}{2},\;\text{for}\;x\in\R,\,y>d_L,
  \end{split}
  \label{TRAP}
\end{equation}
with $r_0$ as in \eqref{second-derW}
\label{trap}
\end{lemma}
\begin{proof}
Set $\bar{y}=\frac{1}{k}\log{\frac{4K}{r_0}}$, then from
(\ref{bar-exp}) it follows
\begin{equation}
  \begin{split}
    &\vert\bar{u}(y)-a_-\vert\leq\frac{r_0}{4},\;\text{for}\;y\leq-\bar{y},\;\bar{u}\in\{\bar{u}_-,\bar{u}_+\},\\
    &\vert\bar{u}(y)-a_+\vert\leq\frac{r_0}{4},\;\text{for}\;y\geq +\bar{y},\;\bar{u}\in\{\bar{u}_-,\bar{u}_+\}.
  \end{split}
  \label{In-CONV}
\end{equation}

Given $u\in\mathcal{A}^L$, define
\[
X_0=\Bigl\{x\in[0,L]:\|u(x,\cdot)-\bar{u}_\pm(\cdot-r)\|_1\geq\frac{r_0}{8\sqrt{2}},\;r\in\R\Bigr\}.
\]
If $u$ satisfies (\ref{energy-bound1}), then Lemma \ref{away} or
Proposition \ref{away1} implies
\[
\vert X_0\vert\leq\frac{C_0}{e_{\frac{r_0}{8\sqrt{2}}}}.
\]
Therefore for all $L>\frac{C_0}{e_{\frac{r_0}{8\sqrt{2}}}}$ there
exist $\bar{x}\in[0,L]$, $\bar{r}\in\R$ and
$\bar{u}\in\{\bar{u}_-,\bar{u}_+\}$ such that
\begin{equation}
  \|u(\bar{x},\cdot)-\bar{u}(\cdot-\bar{r})\|_1<\frac{r_0}{8\sqrt{2}}.
  \label{LINF}
\end{equation}

Since we have $\mathcal{J}(u_r) = \mathcal{J}(u)$ for $u_r(x,y)=u(x,y+r), r\in\R$,
\eqref{LINF}, we can identify $u_r$ with $u$.
Then (\ref{LINF}) implies, via $\Vert
v\Vert_{L^\infty}\leq\sqrt{2}\Vert v\Vert_1$, the estimate
\begin{equation}
  \|u(\bar{x},\cdot)-\bar{u}\|_{L^\infty(\R;\R^m)}<\frac{r_0}{8}.
  \label{Linfty}
\end{equation}
Consider now the set
\[
Y_0=\Bigl\{y\in\R:\vert
u(x_y,y)-u(\bar{x},y)\vert\geq\frac{r_0}{8},\;\text{for some}\;
x_y\in(\bar{x},\bar{x}+L)\Bigr\}.
\]
For $y\in Y_0$ it results
\[
\begin{split}
  &\frac{r_0}{8}\leq\vert u(x_y,y)-u(\bar{x},y)\vert\leq\vert x_y-\bar{x}\vert^\frac{1}{2}
  \Bigl(\int_{\bar{x}}^{x_y}\vert u_x(x,y)\vert^2 dx\Bigr)^\frac{1}{2}\\
  &\leq L^\frac{1}{2}\Bigl(\int_{\bar{x}}^{\bar{x}+L}\vert u_x(x,y)\vert^2 dx\Bigr)^\frac{1}{2},\\
\end{split}
\]
so that
\[
\frac{r_0^2}{64}\vert Y_0\vert\leq
L\int_\R\int_{\bar{x}}^{\bar{x}+L}\vert u_x(x,y)\vert^2 dx\leq 2LC_0.
\]
It follows
\[
\vert Y_0\vert\leq 128\frac{LC_0}{r_0^2},
\]
therefore there exists an increasing sequence $\{y_j\}\in\R\setminus
Y_0$ such that
\[
\begin{split}
  & y_0=\bar{y},\;\;\;y_j-y_{j-1}>\vert Y_0\vert,\;j=1,2,\dots\\
  &\vert u(x,y_j)-a_+\vert<\frac{r_0}{2},\;\text{for}\;x\in[\bar{x},\bar{x}+L].
\end{split}
\]
This follows from (\ref{In-CONV}) and (\ref{Linfty}). From the proof
of the cut-off lemma in \cite{alikakos-fusco4c4} we infer that, if the
measure of the set
\[
\Bigl\{(x,y) \in[\bar{x},\bar{x}+L]\times[y_{j-1},y_j]:
\vert u(x,y)-a_+\vert>\frac{r_0}{2}\Bigr\}
\]
is positive, then there exists a map
$v_j:\R\times[y_j,y_{j+1}]\rightarrow\R^m$ which is L-periodic in
$x\in\R$, coincides with $u$ on the boundary of the strip
$\R\times(y_j,y_{j+1})$ and satisfies
\begin{equation}
  \mathcal{J}_{\Omega_j}(v_j)<\mathcal{J}_{\Omega_j}(u),
  \label{better}
\end{equation}
where $\Omega_j=(\bar{x},\bar{x}+L)\times(y_j,y_{j+1})$,
$j=1,2,\dots$. From this we see that to each map $u\in\mathcal{A}^L$
that satisfies (\ref{energy-bound1}) but not
\[
  \vert u(x,y)-a_+\vert<\frac{r_0}{2},\;\text{for}\;x\in\R,\,y>\bar{y}+\vert Y_0\vert.
\]
we can associate a map $v$ that satisfies this inequality and
(\ref{better}).  This and a similar argument concerning the other
inequality in (\ref{TRAP}) establish the lemma with $d_L=\bar{y}+\vert
Y_0\vert$.
\end{proof}
With Lemma \ref{trap} at hand the existence of a minimizer
$u^L\in\mathcal{A}^L$ follows by standard variational arguments. The
minimizer $u^L$ satisfies (\ref{u-bound1}). From this, the assumed
smoothness of $W$ and elliptic theory it follows
\begin{equation}
  \|u^L\|_{C^{2,\beta}(\R^2;\R^m)}\leq C^*,
  \label{C2alpha}
\end{equation}
for some constants $C^*>0$, $\beta\in(0,1)$ independent of $L$ and
$u^L$ is a classical solution of (\ref{elliptic}). Moreover, from the
fact that $u^L$ satisfies (\ref{TRAP}) and a comparison argument we
obtain
\begin{equation}
  \begin{split}
    &\vert u(x,y)-a_-\vert\leq Ke^{-k(\vert y\vert-d_L)},\;\text{for}\;x\in\R,\,y<-d_L,\\
    &\vert u(x,y)-a_+\vert\leq Ke^{-k(\vert y\vert-d_L)},\;\text{for}\;x\in\R,\,y>d_L.
  \end{split}
  \label{LExp}
\end{equation}
and, for $\alpha=(\alpha_1,\alpha_2)$, $\alpha_i=1,2$, $\vert\alpha\vert=1,2$
\[
  \vert D^\alpha u^L(x,y)\vert\leq Ke^{-k(\vert
    y\vert-d_L)},\;\text{for}\;\vert y\vert>d_L.
\]

\subsection{Basic lemmas}
\vskip.2cm To show that the minimizer $u^L$ has the properties listed
in Theorem \ref{periodic1}, in particular {\emph{(i)}}, \emph{(vii)}
and {\emph{(viii)}}, we need point-wise estimates on $u^L$ that do not
depend on $L$. For example to prove {\emph{(i)}} we need to show that
$d_L$ in (\ref{LExp}) can be taken independent of $L$. For
\emph{(vii)} and {\emph{(viii)}} a detailed analysis of the behavior
of the trace $u^L(x,\cdot)$ as a function of $x\in(0,L)$ is
necessary. To complete this program we use several ingredients: a
decomposition of $u^L(x,\cdot)$ that we discuss next; two Hamiltonian
identities that, together with the decomposition of $u^L(x,\cdot)$,
allow a representation of the energy
$\mathcal{J}_{(0,L)\times\R}(u^L)$ with a one dimensional integral in
$x$ (see Lemma \ref{properties} and Lemma \ref{represent}) and an
analysis of the behavior of the \emph{effective potential}
$J_\R(\bar{u}+v)-J_\R(\bar{u})$, $\bar{u}\in\{\bar{u}_-,\bar{u}_+\}$
as a function of $v\in H^1(\R;\R^m)$ that we present in Lemma
\ref{lemmaw-true} and in Lemma \ref{away}.

Let $\bar{\mathrm{u}}:\R\rightarrow\R^m$ be a smooth map with the same
asymptotic behavior as $\bar{u}_\pm$.  Set $H^0(\R;\R^m)=L^2(\R;\R^m)$
and let $H^1(\R;\R^m)$ be the standard Sobolev space. For $j=0,1$ let
$\langle\cdot,\cdot\rangle_j$ be the inner product in $H^j(\R;\R^m)$
and $\|\cdot\|_j$ the associated norm. If there is no risk of
confusion, for $j=0$ we simply write $\langle\cdot,\cdot\rangle$ and
$\|\cdot\|$ instead of $\langle\cdot,\cdot\rangle_0$ and
$\|\cdot\|_0$. Set
\[\mathcal{H}^j=\bar{\mathrm{u}}+H^j(\R;\R^m),\]

Define
\[
  q_j^u=\inf_{r\in\R,\pm}\|u-\bar{u}_\pm(\cdot-r)\|_j,\;\;u\in\mathcal{H}^j.
\]
Note that for large $\vert r\vert$ we have
\[
\|u-\bar{u}_\pm(\cdot-r)\|_j\geq \frac{1}{2}\vert
a_+-a_-\vert\sqrt{|r|}.
\]
This and the fact that $\|u-\bar{u}_\pm(\cdot-r)\|_j$ is continuous in
$r$ imply the existence of $h_j\in\R$ and
$\bar{u}_j\in\{\bar{u}_-,\bar{u}_+\}$ such that
\[
q_j^u=\|u-\bar{u}_j(\cdot-h_j)\|_j.
\]
$q_j^u$ is a continuous function of $u\in\mathcal{H}^j$ and a standard
argument implies that
\begin{equation}
  \langle u-\bar{u}_j(\cdot-h_j),\bar{u}_j^\prime(\cdot-h_j)\rangle_j=0.
  \label{orthogonal}
\end{equation}
Note that $\bar{u}_j$ remains equal to some fixed
$\bar{u}\in\{\bar{u}_-,\bar{u}_+\}$ while $u$ changes continuously in
the subset of $\mathcal{H}^j$ where
\[
  q_j^u<\frac{1}{2}\inf_{r\in\R}\|\bar{u}_+-\bar{u}_-(\cdot-r)\|_0.
\]
We quote from Section 2 in \cite{scha}
\begin{lemma}
  There exists $\bar{q}>0$ such that $q_j^u<\bar{q}$ implies that
  $u_j$ and $h_j$ are uniquely determined. Moreover $h_j$ is a function
  of class $C^{3-j}$ of $u\in\mathcal{H}^j$ and
  \begin{equation}
    \begin{split}
      &(D_{{u}}h_j)w=-\frac{\langle
        w,\bar{u}^\prime(\cdot-h_j)\rangle_j}{\|\bar{u}^\prime\|_j^2-\langle
        {{u}}-\bar{u}(\cdot-h_j),\bar{u}^{\prime\prime}(\cdot-h_j)\rangle_j}.
    \end{split}
    \label{derivative}
  \end{equation}
  There are constants $C, \tilde{C}>0$ such that, for $q_1^u<\bar{q}$,
  \begin{equation}
    \begin{split}
      &\vert h_0-h_1\vert\leq Cq_1^u,\\
      &\|u-\bar{u}(\cdot-h_0)\|_1\leq\tilde{C}q_1^u.
    \end{split}
    \label{h-h}
  \end{equation}
  \label{lemmaw}
\end{lemma}
In the following we drop the subscript $0$ and write simply $q^u$,
$\|\cdot\|$, etc. instead of $q_0^u$, $\|\cdot\|_0$, etc.

From Lemma \ref{lemmaw} and (\ref{orthogonal}) it follows that
$u\in\mathcal{H}$ can be decomposed in the form
\begin{equation}
  \begin{split}
    &u=\bar{u}(\cdot-h)+v(\cdot-h),\\
    &\langle v,\bar{u}^\prime\rangle=0,
  \end{split}
  \label{decomposition}
\end{equation}
for some $h\in\R$ and $\bar{u}\in\{\bar{u}_-,\bar{u}_+\}$ and that,
provided $q^u<\bar{q}$, $h\in\R$ and $\bar{u}$ are uniquely
determined. Note that from (\ref{decomposition}) we have
\[
v(s)=u(s+h)-\bar{u}(s)
\]
and
\[
\|v\|=q^u.
\]
In particular the decomposition (\ref{decomposition}) applies to the
minimizer $u^L\in\mathcal{A}^L$:

\begin{equation}
  \begin{split}
    &u^L(x,\cdot)=\bar{u}(\cdot-h^L(x))+v^L(x,\cdot-h^L(x)),\\
    &\langle v^L(x,\cdot),\bar{u}^\prime\rangle = 0,\\
  \end{split}
  \label{decompositionL}
\end{equation}
for some $\bar{u}\in\{\bar{u}_-,\bar{u}+\}$.  Given $x\in\R$ we set
$q^L(x)=q^{u^L(x,\cdot)}$ and $q_1^L(x)=q_1^{u^L(x,\cdot)}$ and recall
that
\[
q^L(x)=\|v^L(x,\cdot)\|=\|u^L(x,\cdot)-\bar{u}(\cdot-h^L(x))\|.
\]
In general $h^L(x)$ is not uniquely determined if $q^L(x)$ is not
sufficiently small. In the following, if there is no risk of confusion,
we drop the superscript $L$ and write simply $q(x)$, $v(x,y)$, $h(x)$,
etc.. instead of $q^L(x)$, $v^L(x,y)$, $h^L(x)$, etc..

From the minimality of $u=u^L$ and its smoothness properties
established in (\ref{C2alpha}) and (\ref{LExp}) it follows that $u^L$
satisfies two Hamiltonian identities. This is the content of the
following lemma, where $c_0$ is defined in \eqref{c0def}.
\begin{lemma}
  Set $u=u^L$. Then there exist constants $\omega$ and $\tilde{\omega}$
  such that, for $x\in\R$, it results
  \begin{equation}
    \int_\R\frac{1}{2}\vert u_x(x,y)\vert^2{d} y=\int_\R\Big(W(u(x,y))+\frac{1}{2}\vert u_y(x,y)\vert^2\Big){d} y-c_0-\omega
    \label{hamilton}
  \end{equation}
  and
  \begin{equation}
    \int_\R u_x(x,y)\cdot u_y(x,y){d} y=\tilde{\omega},\quad\text{ for }\;x\in\R.\hskip2cm
    \label{uxuyconst}
  \end{equation}
  Moreover it results
  \begin{equation}
    \begin{split}
      &\omega=\int_\R\Big(W(u(\frac{L}{4},y))+\frac{1}{2}\vert u_y(\frac{L}{4},y)\vert^2)\Big){d} y-c_0\geq 0,\\
      &\tilde{\omega}=0.
    \end{split}
    \label{omega}
  \end{equation}
  \label{properties}
\end{lemma}
\begin{proof}
The identities (\ref{hamilton}) and (\ref{uxuyconst}) are well known,
see for instance \cite{scha} or \cite{f}. To prove (\ref{omega}) we
observe that $u(\frac{L}{4}-x,y)=u(\frac{L}{4}+x,y)$ implies
$u_x(\frac{L}{4},y)=0$.
\end{proof}
\begin{lemma}
  The constant $\bar{q}$ in Lemma \ref{lemmaw} can be chosen such that, if
  \begin{equation}
    0<q(x)\leq q_1(x)\leq \bar{q},\;\;x\in I,
    \label{stimediq}
  \end{equation}
  for some interval $I\subset\R$, then, for $x\in I$ the maps
  $h(x)=h^L(x)$, $v(x,y)=v^L(x,y)$ $\bar{u}\in\{\bar{u}_-,\bar{u}_+\}$
  in the decomposition (\ref{decompositionL}) are uniquely determined
  and are smooth functions of $x\in I$.  With
  $\nu(x,\cdot)=\nu^L(x,\cdot)$ defined by
  $v(x,\cdot)=q(x)\nu(x,\cdot)$, it results
  \begin{equation}
    h^\prime(x)=\frac{\langle v_x(x,\cdot),v_y(x,\cdot)\rangle}{\|\bar{u}^\prime+v_y(x,\cdot)\|^2}
    =\frac{q^2(x)\langle\nu_x(x,\cdot),\nu_y(x,\cdot)\rangle}{\|\bar{u}^\prime+q(x)\nu_y(x,\cdot)\|^2},
    \label{hprime}
  \end{equation}
  and
  \begin{equation}
    \begin{split}
      &\|u_x(x,\cdot)\|^2=\|v_x(x,\cdot)\|^2-\frac{\langle v_x(x,\cdot),v_y(x,\cdot)\rangle^2}{\|\bar{u}^\prime+v_y(x,\cdot)\|^2}\\
      &=q^\prime(x)^2+q^2(x)\|\nu_x(x,\cdot)\|^2
      -q^4(x)\frac{\langle\nu_x(x,\cdot),\nu_y(x,\cdot)\rangle^2}{\|\bar{u}^\prime+q(x)\nu_y(x,\cdot)\|^2}.
    \end{split}
    \label{xenergy}
  \end{equation}
  Moreover the map
  \[
  (0,q(x)]\ni p\rightarrow f(p,x)\|\nu_x(x,\cdot)\|^2:=
    p^2\|\nu_x(x,\cdot)\|^2-p^4\frac{\langle\nu_x(x,\cdot),\nu_y(x,\cdot)\rangle^2}{\|\bar{u}^\prime+p\nu_y(x,\cdot)\|^2}
  \]
  is non-negative and non-decreasing for each fixed $x\in I$.
  \label{represent}
\end{lemma}
\begin{proof}
From (\ref{decompositionL}) with $u=u^L$, $v=v^L$ we obtain
\[
\begin{split}
  & u_x(x,\cdot)=-h^\prime(x)\Big(\bar{u}^\prime(\cdot-h(x))+v_y(x,\cdot-h(x))\Big)+v_x(x,\cdot-h(x)),\\
  & u_y(x,\cdot)=\bar{u}^\prime(\cdot-h(x))+v_y(x,\cdot-h(x)).
\end{split}
\]
and therefore Lemma \ref{properties} and (\ref{decompositionL}) that
implies
\[
\langle v_x(x,\cdot),\bar{u}^\prime\rangle=0,\;\;x\in I,
\]
yield
\begin{equation}
  \begin{split}
    &0=\langle u_x(x,\cdot),u_y(x,\cdot)\rangle=
    -h^\prime(x)\|\bar{u}^\prime+v_y(x,\cdot)\|^2+\langle v_x(x,\cdot),v_y(x,\cdot)\rangle.
  \end{split}
  \label{hprime1}
\end{equation}
From assumption \eqref{stimediq} and (\ref{h-h}) we have
$\|v_y(x,\cdot)\|\leq\|v\|_1\leq\tilde{C}
q_1(x)\leq\tilde{C}\bar{q}$ and
$\bar{q}\leq\frac{\|\bar{u}^\prime\|}{2\tilde{C}}$ implies
\[
  \frac{1}{2}\|\bar{u}^\prime\|\leq\|\bar{u}^\prime+v_y(x,\cdot)\|\leq\frac{3}{2}\|\bar{u}^\prime\|.
\]
Therefore (\ref{hprime1}) can be solved for $h^\prime(x)$ and the
first expression of $h^\prime(x)$ in (\ref{hprime}) is
established. For the other expression we observe that $\langle
v_x,v_y\rangle =\langle
q_x\nu+q\nu_x,q\nu_y\rangle=q^2\langle\nu_x,\nu_y\rangle$ that follows
from $\langle\nu(x,\cdot-r),\nu(x,\cdot-r)\rangle=1$, for $r\in\R$
which implies $\langle\nu_y(x,\cdot),\nu(x,\cdot)\rangle=0$.  A
similar computation that also uses (\ref{hprime}) yields
(\ref{xenergy}).

It remains to prove the monotonicity of $p\rightarrow
f(p,x)\|\nu_x(x,\cdot)\|^2$. We can assume $\|\nu_x\|>0$ otherwise
there is nothing to be proved. We have
\[
p\|\nu_y(x,\cdot)\|\leq
q(x)\|\nu_y(x,\cdot)\|=\|v_y(x,\cdot)\|\leq\tilde{C}\bar{q},
\]
and therefore
\[
\begin{split}
  &D_pf(p,\cdot)=2p-4p^3\frac{\langle\frac{\nu_x}{\|\nu_x\|},\nu_y\rangle^2}{\|\bar{u}^\prime+p\nu_y\|^2}
  +2p^4\frac{\langle\frac{\nu_x}{\|\nu_x\|},\nu_y\rangle^2\langle\bar{u}^\prime+p\nu_y,\nu_y\rangle}{\|\bar{u}^\prime+p\nu_y\|^4}\\
  &\geq 2p\Big(1-2\frac{(\tilde{C}\bar{q})^2}{(\|\bar{u}^\prime\|-\tilde{C}\bar{q})^2}
  -\frac{(\tilde{C}\bar{q})^3}{(\|\bar{u}^\prime\|-\tilde{C}\bar{q})^3}\Big).
\end{split}
\]
This proves $D_pf(p,\cdot)>0$ for
$\bar{q}\leq\frac{\|\bar{u}^\prime\|}{3\tilde{C}}$.  The proof is
complete.
\end{proof}
Next we list some properties of the \emph{ effective potential }
$J_\R(u)-c_0$ that depend on the decomposition (\ref{decomposition})
of $u$. Define
\[
  \mathcal{W}(v)=J_\R(\bar{u}+v)-J_\R(\bar{u}).
\]
where $v$ is as in (\ref{decomposition}) and $u\in\mathcal{H}^1$.  If
we set $v=q\nu$, with $q=\|v\|\neq 0$, $\mathcal{W}$ can be considered
as a function of $q\in\R$ and $\nu\in H^1(\R;\R^m)$, $\|\nu\|=1$. We
have (see \cite{f})
\begin{lemma}
  Assume that $\vert v^\prime\vert\leq C$ for some $C>0$. Then
  \begin{equation}
    \|v\|_{L^\infty(\R;\R^m)}\leq C_1\|v\|^\frac{2}{3},
    \label{linfty}
  \end{equation}
  for some $C_1>0$. The constant $\bar{q}>0$ in Lemma \ref{lemmaw} can
  be chosen such that the effective potential $\mathcal{W}(q\nu)$ is
  increasing in $q$ for $q\in[0,\bar{q}]$ and there is $\mu>0$ such that
  \begin{equation}
    \frac{\partial^2}{\partial q^2}\mathcal{W}(q\nu)\geq\mu(1+\|\nu^\prime\|^2),\quad q\in(0,\bar{q}],
      \label{W2geqq2}
  \end{equation}
  and
  \[
    \begin{split}
      &\mathcal{W}(q\nu)\geq\frac{1}{2}\mu q^2(1+\|\nu^\prime\|^2),\quad q\in(0,\bar{q}]\\
        &\Leftrightarrow\\
        &\mathcal{W}(v)\geq\frac{1}{2}\mu\|v\|_1^2,\quad\|v\|\in(0,\bar{q}].
    \end{split}
  \]
  \label{lemmaw-true}
\end{lemma}

Lemma \ref{lemmaw-true} describes the properties of the effective
potential $\mathcal{W}$ in a neighborhood of one of the connections
$\bar{u}_\pm$. We also need a lower bound for the effective potential
away from a neighborhood of the connections. We have the following
result, see Corollary 3.2 in \cite{scha}) or Proposition \ref{away1}
in the Appendix, where we give an elementary proof.

\begin{lemma}
  For each ${p}>0$ there exists $e_p>0$ such that, if $u\in\mathcal{H}^1$ satisfies
\item
  \[
  q_1^u\geq {p},
  \]
  then
  \[
  J_\R(u) - c_0\geq e_{p}.
  \]
  Moreover $e_{p}$ is continuous in $p$ and for $p\leq\|v\|_1$, $\|v\|_1$ small, it results
  \begin{equation}
    e_{p}\leq J_\R(\bar{u}+v)-c_0\leq C^1\|v\|_1^2,\;\;
    v\in H^1(\R;\R^m),\;\bar{u}\in\{\bar{u}_-,\bar{u}_+\},
    \label{J-upper0}
  \end{equation}
  with $C^1>0$ a constant.
  \label{away}
\end{lemma}
Set $u=u^L$ and let
\[
  p\in(0,\bar{q}),
\]
be a number to be chosen later. From (\ref{J-upper0}) there is $p^*<p$
such that $e_{p^*}<e_p$. Let $S_{p^*}\subset[0,L]$ be the complement
of the set
\[
\tilde{S}_{p^*}=\{x\in[0,L]:J_\R(u(x,\cdot))-c_0>e_{p^*}\}.
\]
From (\ref{energy-bound1}) we have
\[
e_{p^*}\vert\tilde{S}_{p^*}\vert\leq\int_0^L\Big(J_\R(u(x,\cdot))-c_0\Big)dx\leq
C_0,
\]
which implies
\[
\vert\tilde{S}_{p^*}\vert\leq\frac{C_0}{e_{p^*}},\;\;\vert{S}_{p^*}\vert\geq
L-\frac{C_0}{e_{p^*}}.
\]
For $x\in{S}_{p^*}$ we have $J_\R(u(x,\cdot))-c_0\leq e_{p^*}<e_p$ and
therefore Lemma \ref{away} implies $q_1(x)<p$. It follows $q(x)\leq
q_1(x)\leq \bar{q}$ and Lemma \ref{represent}
implies the uniqueness of the decomposition
(\ref{decompositionL}).  On the other hand Lemma~\ref{lemmaw-true} yields
\begin{equation}
  \|v_y(x,\cdot)\|^2\leq\frac{2}{\mu}\Big(J_\R(u(x,\cdot))-c_0\Big)\leq\frac{2
    e_p}{\mu},\;\;x\in{S}_{p^*}.
  \label{stimavy}
\end{equation}
We fix $p$ so that
\[
\frac{2 e_p}{\mu}\leq\frac{1}{(1+\sqrt{2})^2}\|\bar{u}^\prime\|^2.
\]
With this choice of $p$ we have
\begin{equation}
  \|v_y(x,\cdot)\|^2
  \leq\frac{1}{(1+\sqrt{2})^2}\|\bar{u}^\prime\|^2,\;\;x\in{S}_{p^*}.
  \label{vy-in-sp}
\end{equation}
We also have
\begin{equation}
  \|v_x(x,\cdot)\|^2\leq 4\Big(J_\R(u(x,\cdot))-c_0\Big),\;\;x\in{S}_{p^*}.
  \label{vx-in-sp}
\end{equation}
To see this, note that from (\ref{hamilton}) and $\omega\geq 0$ it follows
\begin{equation}
  \frac{1}{2}\|u_x(x,\cdot)\|^2\leq J_\R(u(x,\cdot))-c_0,\;\;x\in[0,L],
  \label{ux-hamilton}
\end{equation}
and that from (\ref{xenergy}) and (\ref{vy-in-sp}) it follows
\[
\|v_x(x,\cdot)\|^2\leq 2\|u_x(x,\cdot)\|^2,\;\;x\in{S}_{p^*}.
\]
From (\ref{hprime}), \eqref{stimavy}, (\ref{vy-in-sp}) and
(\ref{vx-in-sp}) we obtain
\begin{equation}
  \int_{{S}_{p^*}}\vert h^\prime(x)\vert
  dx\leq\frac{\sqrt{2}(1+\sqrt{2})^2}{\sqrt{\mu}\Vert\bar
    u'\Vert^2}\int_{{S}_{p^*}}\Big(J_\R(u(x,\cdot))-c_0\Big)dx
  \leq\frac{\sqrt{2}(1+\sqrt{2})^2}{\sqrt{\mu}\Vert\bar u'\Vert^2}C_0.
  \label{int-hprime}
\end{equation}
\begin{lemma}
  There is a constant $C_h>0$ independent of $L$ such that
  \[
  \vert h(x)-h(x')\vert\leq C_h,\;\;x,x'\in{S}_{p^*}.
  \]
  \label{h-bounded}
\end{lemma}
\begin{proof}
$\tilde{S}_{p^*}$ is the union of a countable family of intervals
  $\tilde{S}_{p^*}=\cup_j(\alpha_j,\beta_j)$.  Therefore, for each
  $x,x'\in{S}_{p^*}$ we have
 \[
 \vert h(x)-h(x')\vert\leq\int_{{S}_{p^*}}\vert h^\prime(x)\vert dx+\sum_j|h(\beta_j)-h(\alpha_j)| .
 \]
Since we have already estimated the first term, see
(\ref{int-hprime}), to complete the proof it remains to evaluate the
sum on the right hand side of this inequality.  Set
$\lambda=\frac{\bar{q}^2}{8C_0}$ and let
$I_\lambda=\{j:\beta_j-\alpha_j\leq\lambda\}$ and
$\tilde{I}_\lambda=\{j:\beta_j-\alpha_j>\lambda\}$. Note that
$\tilde{I}_\lambda$ contains at most $\frac{C_0}{e_{p^*}\lambda}$
intervals. For $j\in I_\lambda$ and $x\in(\alpha_j,\beta_j)$ we have
\[
\begin{split}
  &\vert u(x,y)-u(\alpha_j,y)\vert\leq\vert x-\alpha_j\vert^\frac{1}{2}\Big(\int_{\alpha_j}^x\vert u_x(s,y)\vert^2ds\Big)^\frac{1}{2},\\
  &\|u(x,\cdot)-u(\alpha_j,\cdot)\|^2\leq 2\lambda C_0\leq\frac{\bar{q}^2}{4}.
\end{split}
\]
From this and $\alpha_j\in{S}_{p^*}$, that implies
\[
q(\alpha_j)=\|u(\alpha_j,\cdot)-\bar{u}(\cdot-h(\alpha_j))\|<p\leq\frac{\bar{q}}{2},
\]
we conclude that
\[
q(x) \leq\|u(x,\cdot)-\bar{u}(\cdot-h(\alpha_j)\|
\leq\|u(x,\cdot)-u(\alpha_j,\cdot)\|+q(\alpha_j)\leq\bar{q}.
\]
This and Lemma \ref{lemmaw} imply that, for $x\in(\alpha_j,\beta_j)$,
with $j\in I_\lambda$, $u=u^L$ can be decomposed as in
(\ref{decompositionL}) and that
$h^\prime(x)=(D_uh)u_x(x,\cdot)$. Therefore from (\ref{derivative})
and assuming as we can
$\bar{q}\leq\frac{\|\bar{u}^\prime\|^2}{2\|\bar{u}''\|}$ we have
\[
\vert h^\prime(x)\vert\leq
2\frac{\|u_x(x,\cdot)\|}{\|\bar{u}^\prime\|},\;\;x\in(\alpha_j,\beta_j),\;j\in
I_\lambda.
\]
It follows
\[
\begin{split}
&\sum_{j\in I_\lambda}\vert h(\beta_j)-h(\alpha_j)\vert\leq
\int_{\cup_{j\in I_\lambda}(\alpha_j,\beta_j)}\vert{h}^\prime(x)\vert {d} x
\leq \frac{2}{\|\bar{u}^\prime\|}\int_{\cup_{j\in I_\lambda}(\alpha_j,\beta_j)}\|u_x(x,\cdot)\| {d} x\\
&\leq \frac{2}{\|\bar{u}^\prime\|}\vert\tilde{S}_{p^*}\vert^\frac{1}{2}\Bigl(\int_0^L\|u_x\|^2 {d} x\Bigr)^\frac{1}{2}\leq (2C_0)^\frac{1}{2}\frac{2}{\|\bar{u}^\prime\|}\vert\tilde{S}_{p^*}\vert^\frac{1}{2}.
\end{split}
\]

Assume now $j\in\tilde{I}_\lambda$ and observe that there is a number
$\bar{y}>0$ such that, if $r\geq 2\bar{y}$ and
$y\in[\bar{y},r-\bar{y}]$ or if $r\leq -2\bar{y}$ and
$y\in[r+\bar{y},-\bar{y}]$, it results for
$\mathrm{sg},\tilde{\mathrm{sg}}\in\{-,+\}$
\begin{equation}
\vert\bar{u}_{\mathrm{sg}}(y)-\bar{u}_{\tilde{\mathrm{sg}}}(y-r)\vert\geq\frac{1}{2}\vert
a_+-a_-\vert.
\label{elem-observ}
\end{equation}

Consider first the indices $j\in\tilde{I}_\lambda$ such that $\vert
h(\beta_j)-h(\alpha_j)\vert\leq 4\bar{y}$. We have
\[
  \sum_{j\in\tilde{I}_\lambda,\vert
    h(\beta_j)-h(\alpha_j)\vert\leq4\bar{y}}\vert
  h(\beta_j)-h(\alpha_j)\vert\leq4\bar{y}
  \frac{C_0}{e_{p^*}\lambda}=\frac{32C_0^2}{e_{p^*}\bar{q}^2}\bar{y}.
\]

Let $(\alpha,\beta)$ be one of the intervals $(\alpha_j,\beta_j)$
corresponding to $j\in\tilde{I}_\lambda$ with $\vert
h(\beta_j)-h(\alpha_j)\vert>4\bar{y}$.  If $r>4\bar{y}$ the interval
$(\bar{y},r-\bar{y})$ (if $r<-4\bar{y}$ the interval
$(r+\bar{y},-\bar{y})$) has measure larger then $\frac{\vert
  r\vert}{2}$. This and the assumptions on $(\alpha,\beta)$ imply that
there exist $y^0,y^1\in(\alpha,\beta)$, that satisfy $y^1-y^0=\vert
h(\beta_j)-h(\alpha_j)\vert/2$ and are such that
\begin{equation}
  \vert u(\beta,y)-u(\alpha,y)\vert\geq\frac{1}{4}\vert a_+-a_-\vert,
  \quad\text{ for }\;\;y\in(y^0,y^1).
  \label{bigger-gamma}
\end{equation}
This, provided $p>0$ is sufficiently small, follows from
(\ref{elem-observ}). Indeed we have

\[
  \begin{split}
    &\vert u(\beta,y)-u(\alpha,y)\vert
    \geq\vert\bar{u}_{\mathrm{sg}(\beta)}(y-h(\beta))-\bar{u}_{\mathrm{sg}(\alpha)}(y-h(\alpha))\vert\\
    &-\vert u(\beta,y)-\bar{u}_{\mathrm{sg}(\beta)}(y-h(\beta))\vert
    -\vert u(\alpha,y)-\bar{u}_{\mathrm{sg}(\alpha)}(y-h(\alpha))\vert\\
    &\geq\frac{1}{2}\vert a_+-a_-\vert
    -\vert v(\beta,y-h(\beta))\vert
    -\vert v(\alpha,y-h(\alpha))\vert\\
    &\geq\frac{1}{2}\vert a_+-a_-\vert-C_1(q(\alpha)^\frac{2}{3}+q(\beta)^\frac{2}{3})\\
    &\geq\frac{1}{2}\vert a_+-a_-\vert-2C_1p^\frac{2}{3}\geq\frac{1}{4}\vert a_+-a_-\vert
    ,\quad\text{ for }\;\;y\in(y^0,y^1),
  \end{split}
\]
where we denoted by $\bar{u}_{\mathrm{sg}(x)}$ the map
$\bar{u}\in\{\bar{u}_-,\bar{u}_+\}$ corresponding to
$x\in{S}_{e_{p^*}}$ and we used (\ref{linfty}) based on
\[
|v_y(x,y)|\leq C,
\]
that follows from \eqref{C2alpha} and \eqref{bar-exp}.
Integrating (\ref{bigger-gamma}) in $(y^0,y^1)$ yields
\[
\begin{split}
&\frac{\vert a_+-a_-\vert}{8}\vert h(\beta)-h(\alpha)\vert\leq\int_{y^0}^{y^1}\vert u(\beta,y)-u(\alpha,y)\vert{d} y
\leq\int_{y^0}^{y^1}\int_{\alpha}^{\beta}\vert u_x\vert{d} x{d} y\\
&\leq\frac{1}{\sqrt{2}}\vert h(\beta)-h(\alpha)\vert^\frac{1}{2}(\beta-\alpha)^\frac{1}{2}
\Big(\int_{y^0}^{y^1}\int_{\alpha}^{\beta}\vert u_x\vert^2{d} x{d} y\Big)^\frac{1}{2}\\
&\leq\vert h(\beta)-h(\alpha)\vert^\frac{1}{2}(\beta-\alpha)^\frac{1}{2}C_0^\frac{1}{2}.
\end{split}
\]
It follows $\vert h(\beta)-h(\alpha)\vert\leq\frac{64 C_0}{\vert
  a_+-a_-\vert^2}(\beta-\alpha)$ and in turn
\[
  \begin{split}
    &\sum_{j\in\tilde{I}_\lambda,\vert
      h(\beta_j)-h(\alpha_j)\vert>4\bar{y}}\vert
    h(\beta_j)-h(\alpha_j)\vert\\
    &\leq\frac{64 C_0}{\vert
      a_+-a_-\vert^2}\sum_{j\in\tilde{I}_\lambda,\vert
      h(\beta_j)-h(\alpha_j)\vert>4\bar{y}}(\beta_j-\alpha_j)\leq\frac{64
      C_0}{\vert a_+-a_-\vert^2}\vert\tilde{S}_{p^*}\vert.
  \end{split}
\]
The proof is complete.
\end{proof}
With Lemma \ref{h-bounded} at hand we can show that $d_L$ in
(\ref{TRAP}) can be taken independent of $L$ and that $u=u^L$
converges to $a_\pm$ as $y\rightarrow\pm\infty$ uniformly in
$x\in\R$.

Next we prove that the restriction $x,x'\in S_{p^*}$ in
Lemma~\ref{h-bounded} can be removed.  We have indeed
\begin{lemma}
  There is a constant $C_h>0$ independent of $L$ such that
  \[
  \vert h(x)-h(x')\vert\leq C_h,\;\;x,x'\in[0,L].
  \]
  \label{h-bounded-all}
\end{lemma}
\begin{proof}
Assuming that $p>0$ is sufficiently small, from (\ref{linfty}) we have
\begin{equation}
  \vert u(x,y)-\bar{u}_{\mathrm{sg}(x)}(y-h(x))\vert\leq C^1
  p^\frac{2}{3}\leq\frac{r_0}{8},\;\;x\in S_{p^*}, y\in\R,
\label{LinftyB}
\end{equation}
where $r_0$ is defined in \eqref{second-derW}.  By
Lemma~\eqref{h-bounded} there exist $h_+, h_-$ such that
$2\delta_h:=h_+-h_-$ is bounded independently of $L$ and
\[
\begin{split}
&\vert\bar{u}_{\mathrm{sg}(x)}(y-h(x))-a_+\vert\leq\frac{r_0}{8},\;\;y\geq
  h_+,\;x\in S_{p^*},\\
&\vert\bar{u}_{\mathrm{sg}(x)}(y-h(x))-a_-\vert\leq\frac{r_0}{8},\;\;y\leq
h_-,\;x\in S_{p^*}.
\end{split}
\]
The first relation and (\ref{LinftyB}) imply
\begin{equation}
  \vert u(x,y)-a_+\vert\leq\frac{r_0}{4},\;\;y\geq h_+,\;x\in S_{p^*}.
  \label{near-a-Sp}
\end{equation}
Now define $Y\subset\R$ by setting
\[
Y=\{y> h_+:\exists\, x_y\in[0,L]\; \text{such that}\; \vert
u(x_y,y)-a_+\vert\geq\frac{r_0}{2}\},
\]
From (\ref{near-a-Sp}) it follows that $y\in Y$ implies that $x_y$
belongs to $\tilde{S}_{p^*}$ and therefore to one of the intervals,
say $(\alpha,\beta)$, that compose $\tilde{S}_{p^*}$. From
(\ref{near-a-Sp}) with $x=\alpha$ it follows $\vert
u(x_y,y)-u(\alpha,y)\vert\geq\frac{r_0}{4}$ for $y\in Y$, and
therefore we have
\[
\frac{r_0}{4}\leq\int_\alpha^{x_y}\vert u_x(x,y)\vert {d} x
\leq\vert\beta-\alpha\vert^\frac{1}{2}\Big(\int_\alpha^\beta\vert
u_x(x,y)\vert^2 {d} x\Big)^\frac{1}{2},\qquad y\in Y,
\]
and in turn
\[
\vert
Y\vert\frac{r_0^2}{16}\leq\vert\tilde{S}_{p^*}\vert\int_{\tilde{S}_{p^*}}\int_\alpha^\beta\vert
u_x(x,y)\vert^2{d} x{d} y\leq 2C_0\vert\tilde{S}_{p^*}\vert,
\]
and we see that the measure of $Y$ is bounded independently of
$L$. Then there exists an increasing sequence $y_j\rightarrow+\infty$
such that
\[
\begin{split}
  &y_1\leq h_++2\vert Y\vert,\\
  &\vert u(x,y_j)-a_+\vert<\frac{r_0}{2},\;\;x\in[0,L],\;\;j=1,\ldots.
\end{split}
\]
This and the cut-off lemma in \cite{af3} imply
\[
\vert u(x,y)-a_+\vert\leq\frac{r_0}{2},\;\;y\geq h_++2\vert Y\vert,\;x\in[0,L].
\]
A similar argument yields
\[
\vert u(x,y)-a_-\vert\leq\frac{r_0}{2},\;\;y\leq h_--2\vert Y\vert,\;x\in[0,L].
\]
The lemma follows from these relations and the fact that $h_+-h_-$ and
$|Y|$ do not depend on $L$.

\end{proof}

\begin{corollary}
We can assume that the minimizer $u^L$ satisfies
  \begin{equation}
    \begin{split}
      &\vert u^L(x,y)-a_+\vert\leq Ke^{-k y},\;\;y>0,\;x\in\R,\\
      &\vert u^L(x,y)-a_-\vert\leq Ke^{k y},\;\;y<0,\;x\in\R.
    \end{split}
    \label{Exp-ind-L}
  \end{equation}
  and
  \begin{equation}\label{DExp-ind-L}
    \vert D^\alpha u^L(x,y)\vert\leq Ke^{-k\vert y\vert},\;\;y\in\R,\;x\in\R,
  \end{equation}
  for $\alpha=(\alpha_1,\alpha_2)$, $\vert\alpha\vert=1,2$, with
  constants $k, K>0$ independent of $L$.
  \label{exp-ind-L}
\end{corollary}
\begin{proof}
Using again the translation invariance of the energy $\mathcal{J}$, by
identifying $u(x,y)$ with $u_{\delta_h}(x,y)=u(x,y+\delta_h)$, we can
assume that the minimizer $u$ satisfy
\[
\begin{split}
  &\vert u(x,y)-a_+\vert\leq\frac{r_0}{2},\;\;y\geq \delta_h + 2\vert Y\vert,\;x\in[0,L]\\
&\vert u(x,y)-a_-\vert\leq\frac{r_0}{2},\;\;y\leq -\delta_h - 2\vert Y\vert,\;x\in[0,L].
\end{split}
\]
These inequalities and a standard argument, based on the
non-degeneracy of $a_+,a_-$, imply (\ref{Exp-ind-L}). Inequality
\eqref{DExp-ind-L} follows from (\ref{Exp-ind-L}) and elliptic
regularity.  The proof is complete.

\end{proof}

\begin{remark}
From (\ref{Exp-ind-L}) it follows that we have $\vert h(x)\vert\leq
C_h$, for $x\in[0,L]$ with $C_h$ independent of $L$. Note that this is
true in spite of the fact that $h(x)$, if $q(x)$ is large, may be
discontinuous.

The bound on $h(x)$ together with (\ref{Exp-ind-L}),
(\ref{DExp-ind-L}) and (\ref{bar-exp}) imply that
\begin{equation}
v(x,y)=u(x,y+h(x))-\bar{u}_{\mathrm{sg}(x)}(y),
\label{vv}
\end{equation}
and its first and second derivative with respect to $y$ satisfy
exponential estimates of the form
\begin{equation}
  \vert D_y^i v(x,y)\vert\leq Ke^{-k\vert y\vert},\;\;y\in\R,\;x\in\R,\;i=0,1,2.
\label{vDExp-ind-L}
\end{equation}
with constants $k,K>0$ independent of $L$. From this and the identity $\|v_y\|^2+\langle v,v_{yy}\rangle=0$ it follows
\begin{equation}
  \|v_y(x,\cdot)\|\leq C_2 q(x)^\frac{1}{2},
  \label{vy}
\end{equation}
with $C_2>0$ independent of $L$.  This inequality implies that in each
interval where $q(x)\leq q^*$, for some $q^*>0$, we can use the
expressions of $h^\prime(x)$ and $\|u_x(x,\cdot)\|$ in Lemma
\ref{represent} and we have the monotonicity of the function $p\mapsto
f(p,x)$.
\label{q-suff}
\end{remark}

\subsection{Conclusion of the proof of Theorem \ref{periodic1}}
As before we set $u=u^L$. Since $u\in\mathcal{A}^L$ we have in
particular $u(0,y)=\gamma u(0,y)$ that means $u(0,y)\in\pi_\gamma$,
$\pi_\gamma$ the plane fixed by $\gamma$. From this and
$\bar{u}_-=\gamma\bar{u}_+$, $\bar{u}_-\neq\bar{u}_+$ it follows
\[
q_1(0)=\inf_{r\in\R,\pm}\|u(0,\cdot)-\bar{u}_\pm(\cdot-r)\|_1\geq\frac{1}{2}\|\bar{u}_+-\bar{u}_-\|.
\]
We assume that the constant $q^*$ introduced above satisfies
$q^*<\frac{1}{2}\|\bar{u}_+-\bar{u}_-\|$ and set $p=q^*/2$. Then,
provided $L$ is sufficiently large, there exists $x_p>0$ such that
\begin{equation}
  \begin{split}
    &q_1(x)>p,\;\;x\in[0,x_p),\\
      &q_1(x_p)=p.
  \end{split}
  \label{xp}
\end{equation}
Indeed, from Lemma~\ref{away} and \eqref{energy-bound1} it follows
$x_pe_p\leq\int_0^{x_p}(J_\R(u(x,\cdot)-c_0))dx\leq C_0$, so that
\begin{equation}
  x_p\leq l_p:=\frac{C_0}{e_p}.
  \label{xp-b}
\end{equation}
From (\ref{xp}), (\ref{xp-b}) and the symmetry $u(\frac{L}{4}-x,y)=u(\frac{L}{4}+x,y)$ with $x=\frac{L}{4}-x_p$ we obtain
\begin{equation}
  \begin{split}
    &q(x_p)=q(\frac{L}{2}-x_p)\leq q_1(x_p)=p=q^*/2,\\
    &\bar{u}_{\mathrm{sg}(x_p)}=\bar{u}_{\mathrm{sg}(\frac{L}{2}-x_p)},\\
    &h(x_p)=h(\frac{L}{2}-x_p).
  \end{split}
  \label{q12-small}
\end{equation}
We now show, see Lemma \ref{stay-in} below, that the minimality of
$u=u^L$ and (\ref{q12-small}) imply
\[
  q(x)\leq p\;\;,x\in[x_p,\frac{L}{2}-x_p].
\]
In the proof of this fact, for $x$ in certain intervals, we use test
maps of the form
\begin{equation}
  \hat{u}(x,y)=\bar{u}(y-\hat{h}(x))+\hat{q}(x)\nu(x,y-\hat{h}(x))
  \label{test}
\end{equation}
for suitable choices of the functions $\hat{q}=\hat{q}(x)$ and
$\hat{h}=\hat{h}(x)$. We always take $\hat{q}(x)\leq q(x)\leq p$. Note
that in (\ref{test}) the direction vector $\nu(x,\cdot)$ is the one
associated to $v(x,\cdot)=q(x)\nu(x,\cdot)$ with $v(x,\cdot)$ defined
in the decomposition (\ref{decompositionL}) of $u$.

\noindent
From (\ref{test}) it follows
\begin{equation}
  \int_\R\vert\hat{u}_x\vert^2dy=(\hat{h}^\prime)^2\|\bar{u}^\prime+\hat{q}\nu_y\|^2-2\hat{h}^\prime \hat{q}^2\langle\nu_x,\nu_y\rangle+(\hat{q}^\prime)^2+\hat{q}^2\|\nu_x\|^2.
  \label{hatux-square}
\end{equation}
We choose the value of $\hat{h}^\prime$ that minimizes (\ref{hatux-square}) that is
\[
  \hat{h}^\prime=\hat{q}^2\frac{\langle\nu_x,\nu_y\rangle}{\|\bar{u}^\prime+\hat{q}\nu_y\|^2},
\]
then we get
\[
\int_\R\vert\hat{u}_x\vert^2dy=(\hat{q}^\prime)^2+\hat{q}^2\|\nu_x\|^2
-\hat{q}^4\frac{\langle\nu_x,\nu_y\rangle^2}{\|\bar{u}^\prime+\hat{q}\nu_y\|^2}.
\]
Therefore the energy density of the test map $\hat{u}$ is given by
\begin{equation}
\begin{split}
  &\int_\R\frac{1}{2}\vert\hat{u}_x\vert^2dy+\int_\R(W(\hat{u})+\frac{1}{2}\vert\hat{u}_y\vert^2)dy\\
  &=\frac{1}{2}\Big((\hat{q}^\prime)^2+\hat{q}^2\|\nu_x\|^2-\hat{q}^4\frac{\langle\nu_x,\nu_y\rangle^2}{\|\bar{u}^\prime+\hat{q}\nu_y\|^2}\Big)+\mathcal{W}(\hat{q}\nu)
  + c_0.
\end{split}
\label{hatenergy}
\end{equation}
Note that, since we do not change the direction vector $\nu(x,\cdot)$,
this expression is completely determined once we fix the function
$\hat{q}$.
\begin{lemma}
  If $u=u^L$ satisfies (\ref{q12-small}),
  then
  \[
  q(x)\leq p,\;\;x\in\Bigl[x_p,\frac{L}{2}-x_p\Bigr].
  \]
  \label{stay-in}
\end{lemma}
\begin{proof}
Assume instead that $q(x^*)>p$ for some
$x^*\in(x_p,\frac{L}{2}-x_p)$. We can assume that
$q(x^*)=\max_{x\in[x_p,\frac{L}{2}-x_p]}q(x)$.  We show that this
implies the existence of a competing map $\tilde{u}$ with less energy
than $u$. Consider first the case where $q(x^*)\in(p,2p]$. In this
    case we set $\tilde{u}=\hat{u}$ with $\hat{u}$ defined in
    (\ref{test}) and, see Figure \ref{fig2},

\begin{equation}
  \begin{split}
    &\hat{q}(x)=q(x),\;\;\text{ if }\;q(x)\leq p,\\
    &\hat{q}(x)=2p-q(x),\;\;\text{ if }\;q(x)\in(p,2p].
  \end{split}
  \label{compar1}
\end{equation}

\begin{figure}
  \begin{center}
  \begin{tikzpicture}[xscale=1,yscale=1.75]
\draw [->](3,0)--(12,0);
\draw (3,1)--(12,1);
\draw (3,2)--(12,2);
\draw [blue, ultra thick](3,1).. controls (3.25,.5) and(3.75,.5)..(4,1);
\draw [blue, ultra thick](4,1).. controls (4.25,1.5) and(4.75,1.5)..(5,1);
\draw [red, ultra thick](4,1).. controls (4.25,.5) and(4.75,.5)..(5,1);
\draw [blue, ultra thick](5,1).. controls (5.25,.5) and(5.75,.5)..(6,1);
\draw [blue, ultra thick](6,1).. controls (6.25,1.5) and(6.75,1.75)..(7,1.5);
\draw [blue, ultra thick](7,1.5).. controls (7.25,1.25) and(7.75,1.25)..(8,1.5);
\draw [red, ultra thick](6,1).. controls (6.25,.5) and(6.75,.25)..(7,.5);
\draw [red, ultra thick](7,.5).. controls (7.25,.75) and(7.75,.75)..(8,.5);
\draw [blue, ultra thick](11,1).. controls (11.25,.5) and(11.75,.5)..(12,1);
\draw [blue, ultra thick](10,1).. controls (10.25,1.5) and(10.75,1.5)..(11,1);
\draw [red, ultra thick](10,1).. controls (10.25,.5) and(10.75,.5)..(11,1);
\draw [blue, ultra thick](9,1).. controls (9.25,.5) and(9.75,.5)..(10,1);
\draw [blue, ultra thick](8,1.5).. controls (8.25,1.75) and(8.75,1.5)..(9,1);
\draw [red, ultra thick](8,.5).. controls (8.25,.25) and(8.75,.5)..(9,1);
\draw [dotted](3,0)--(3,1);
\draw [dotted](12,0)--(12,1);
\draw [dotted](6.75,0)--(6.75,1.63);
\draw [dotted](7.5,0)--(7.5,2);
\node at (7.5,-.2){$L/4$};
\node at (3,-.2){$x_p$};
\node at (2.75,1){$p$};
\node at (2.75,2){$2p$};
\node at (12,-.2){$L/2-x_p$};
\node at (6.75,-.2){$x^*$};
\draw [->](4.5,.4)--(4.5,.6);
\node at (4.5,.2){$\hat{q}(x)$};
\draw [->](4.5,1.6)--(4.5,1.4);
\node at (4.5,1.7){$q(x)$};
 \end{tikzpicture}
\end{center}
  \caption{The maps $x\rightarrow q(x)$ and $x\rightarrow\hat{q}(x)$ in Lemma \ref{stay-in}, $q(x^*)\leq 2p$}
\label{fig2}
\end{figure}
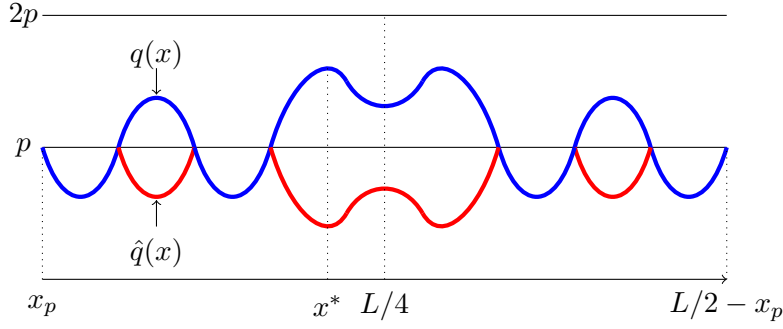
With this definition we have
\begin{equation}
  \tilde{u}(x_p)=u(x_p)=u\Bigl(\frac{L}{2}-x_p\Bigr)=\tilde{u}\Bigl(\frac{L}{2}-x_p\Bigr).
  \label{xp=Lxp}
\end{equation}
To see this we note that
$\max_{x\in[x_p,\frac{L}{2}-x_p]}q(x)=q(x^*)\leq q^*$ implies that
$\mathrm{sg}(x)$ is constant in $[x_p,\frac{L}{2}-x_p]$ therefore from
(\ref{vv}) and $u(x,y)=u(\frac{L}{2}-x,y)$ it follows
\[
v_x(x,y)=-v_x\Bigl(\frac{L}{2}-x,y\Bigr)\;\;\;v_y(x,y)=v_y\Bigl(\frac{L}{2}-x,y\Bigr),
\]
and by consequence
\[
\hat{h}^\prime(x)=-\hat{h}^\prime\Bigl(\frac{L}{2}-x\Bigr),
\]
which yields
\[
\hat{h}\Bigl(\frac{L}{2}-x_p\Bigr) =
h(x_p)+\int_{x_p}^{\frac{L}{2}-x_p}\hat{h}^\prime(x)dx = h(x_p) =
h\Bigl(\frac{L}{2}-x_p\Bigr).
\]
This and $q(x_p)=q(\frac{L}{2}-x_p)$ imply (\ref{xp=Lxp}). It remains
to show that the energy of $\tilde{u}$ is strictly less than the
energy of $u$. By comparing (\ref{hatenergy}) with the analogous
expression of the energy of $u$ and observing that
$(\hat{q}^\prime)^2=(q^\prime)^2$ and $\hat{q}(x)\leq q(x)$ with
strict inequality near $x^*$ we see that this is indeed the case.

Assume now that $q(x^*)>2p$, see Figure \ref{fig3}. Let
$\tilde{x}_p\in(x_p,\frac{L}{4})$ be the number
\[
\tilde{x}_p=\max\{x>x_p: q(s)\leq 2p,\;s\in(x_p,x]\}.
\]
Note that from
$\bar{u}_{\mathrm{sg}(x_p)}=\bar{u}_{\mathrm{sg}(\frac{L}{2}-x_p)}$
and the symmetry of $u$ it follows that $\mathrm{sg}(x)$ is equal to a
constant, say $+$, in
$[x_p,\tilde{x}_p]\cup[\frac{L}{2}-\tilde{x}_p,\frac{L}{2}-x_p]$.

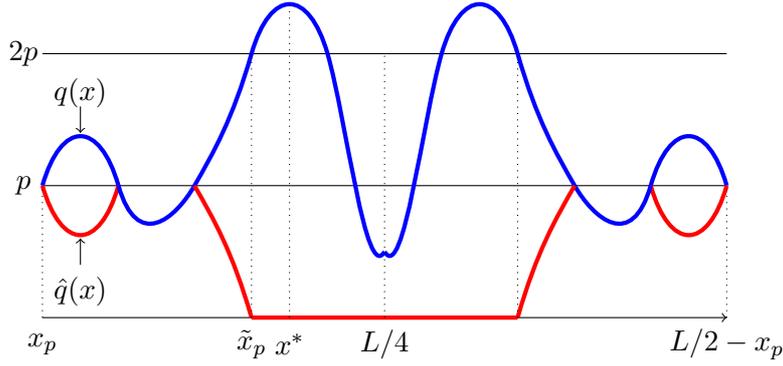
\begin{figure}
  \begin{center}
  \begin{tikzpicture}[xscale=1,yscale=1.75]
\draw [->](3,0)--(12,0);
\draw (3,1)--(12,1);
\draw (3,2)--(12,2);
\draw [blue, ultra thick](3,1).. controls (3.25,1.5) and(3.75,1.5)..(4,1);
\draw [red, ultra thick](3,1).. controls (3.25,.5) and(3.75,.5)..(4,1);
\draw [blue, ultra thick](4,1).. controls (4.25,.5) and(4.75,.75)..(5,1);
\draw [blue, ultra thick](5,1).. controls (5.25,1.25) and(5.50,1.5)..(5.75,2);
\draw [red, ultra thick](5,1).. controls (5.25,.75) and(5.50,.5)..(5.75,0);
\draw [blue, ultra thick](5.75,2).. controls (6,2.5) and(6.50,2.5)..(6.75,2);
\draw [blue, ultra thick](6.75,2).. controls (7,1.5) and(7.25,.25)..(7.5,.5);
\draw [red, ultra thick](5.75,0)--(9.25,0);
\draw [blue, ultra thick](11,1).. controls (11.25,1.5) and(11.75,1.5)..(12,1);
\draw [red, ultra thick](11,1).. controls (11.25,.5) and(11.75,.5)..(12,1);
\draw [blue, ultra thick](10,1).. controls (10.25,.75) and(10.75,.5)..(11,1);
\draw [blue, ultra thick](9.25,2).. controls (9.50,1.5) and(9.75,1.25)..(10,1);
\draw [red, ultra thick](9.25,0).. controls (9.50,.5) and(9.75,.75)..(10,1);
\draw [blue, ultra thick](8.25,2).. controls (8.50,2.5) and(9,2.5)..(9.25,2);
\draw [blue, ultra thick](7.5,.5).. controls (7.75,.25) and(8,1.5)..(8.25,2);
\draw [dotted](3,0)--(3,1);
\draw [dotted](12,0)--(12,1);
\draw [dotted](7.5,0)--(7.5,2);
\draw [dotted](5.75,0)--(5.75,2);
\draw [dotted](9.25,0)--(9.25,2);
\draw [dotted](6.25,0)--(6.25,2.4);
\node at (7.5,-.2){$L/4$};
\node at (5.75,-.2){$\tilde{x}_p$};
\node at (3,-.2){$x_p$};
\node at (2.75,1){$p$};
\node at (2.75,2){$2p$};
\node at (12,-.2){$L/2-x_p$};
\node at (6.25,-.2){$x^*$};
\draw [->](3.5,.4)--(3.5,.6);
\node at (3.5,.2){$\hat{q}(x)$};
\draw [->](3.5,1.6)--(3.5,1.4);
\node at (3.5,1.7){$q(x)$};
\end{tikzpicture}
\end{center}
  \caption{The maps $x\rightarrow q(x)$ and $x\rightarrow\hat{q}(x)$ in Lemma \ref{stay-in}, $q(x^*)>2p$}
\label{fig3}
 \end{figure}

We define the competing map $\tilde{u}$ as follows. In the interval
$[x_p,\tilde{x}_p]$ we set $\tilde{u}=\hat{u}$ with $\hat{q}$ exactly
as in (\ref{compar1}) and
\[
\hat{h}(x)=h(x_p)+\int_{x_p}^x\hat{h}^\prime(s)dx,\;\;x\in[x_p,\tilde{x}_p].
\]
In the interval $(\tilde{x}_p,\frac{L}{2}-\tilde{x}_p)$ we take
\[
\tilde{u}(x,y)=\bar{u}_+(y-\hat{h}(\tilde{x}_p)).
\]
Finally in the interval $[\frac{L}{2}-\tilde{x}_p,\frac{L}{2}-x_p]$ we
set again $\tilde{u}=\hat{u}$ with $\hat{q}$ as in (\ref{compar1}) but
with
\[
\hat{h}(x)=\hat{h}(\tilde{x}_p)+\int_{\frac{L}{2}-\tilde{x}_p}^x\hat{h}^\prime(s)dx,
\;\;x\in\Bigl[\frac{L}{2}-\tilde{x}_p,\frac{L}{2}-x_p\Bigr].
\]
With these definitions $\tilde{u}$ is a continuous piece-wise smooth
map that satisfies (\ref{xp=Lxp}) and, as in the case $q(x^*)\leq 2p$,
one checks that $\tilde{u}$ has energy strictly less than $u$.  The
proof is complete.
\end{proof}

Next we show that the statement of Lemma \ref{stay-in} can be upgraded
to exponential decay. We have indeed
\begin{lemma}
  There exists a positive constants $c^*, C^*$ independent of $L\geq
  L_0$ such that
  \[
  \|v(x,\cdot)\|\leq C^*e^{-c^*x},\;\;x\in\Bigl[0,\frac{L}{4}\Bigr].
  \]
  \label{expD}
\end{lemma}
\begin{proof}
  We show that, under the standing assumption that $2p=q^*>0$ is
  sufficiently small, for $L\geq 4x_p$ it results
  \begin{equation}
    q(x)\leq\sqrt{2}pe^{-\frac{1}{2}\sqrt{\mu}(x-x_p)},\;\;x\in\Bigl[x_p,\frac{L}{4}\Bigr],
    \label{inx>xp}
  \end{equation}
  where $\mu>0$ is the constant in (\ref{W2geqq2}).  Then the lemma
  follows from \eqref{inx>xp} and \eqref{vDExp-ind-L} that implies
  $q(x)=\|v(x,\cdot)\|\leq\frac{K}{\sqrt{k}}$. To prove \eqref{inx>xp}
  we proceed as in the proof of Lemma \ref{lemmaw-true} in
  \cite{f}. We first establish the inequality
  \begin{equation}
    \frac{d^2}{dx^2}\|v(x,\cdot)\|^2\geq\mu\|v(x,\cdot)\|^2,\;\;x\in\Bigl[x_p,\frac{L}{2}-x_p\Bigr].
    \label{ElIn0}
  \end{equation}
  
  We begin by the elementary inequality
  \begin{equation}
    \begin{split}
      &\frac{d^2}{dx^2}\|v(x,\cdot)\|^2=\frac{d^2}{dx^2}\|u(x,\cdot)-\bar{u}_+(\cdot-h(x))\|^2\\
      &\geq2\Big\langle\frac{d^2}{dx^2}\Big(u(x,\cdot)-\bar{u}_+(\cdot-h(x))\Big),u(x,\cdot)-\bar{u}_+(\cdot-h(x))\Big\rangle.
    \end{split}
    \label{ElIn}
  \end{equation}
  From
  \[
  \begin{split}
    &\frac{d^2}{dx^2}\Big(u(x,\cdot)-\bar{u}_+(\cdot-h(x))\Big)\\
    &=u_{xx}(x,\cdot)-\bar{u}_+^{\prime\prime}(\cdot-h(x))(h^\prime(x))^2+\bar{u}_+^\prime(\cdot-h(x))h^{\prime\prime}(x),
  \end{split}
  \]
  and (\ref{ElIn}), using also (\ref{hprime}) (and $\langle\phi,\psi\rangle=\langle\phi(\cdot-r),\psi(\cdot-r)\rangle$), it follows
  \[
    \begin{split}
      &\frac{d^2}{dx^2}\|v(x,\cdot)\|^2\geq
      2\langle u_{xx}(x,\cdot),v(x,\cdot-h(x))\rangle\\
      &-2\langle
      \bar{u}_+^{\prime\prime},v(x,\cdot)\rangle\frac{\langle
        v_x(x,\cdot),v_y(x,\cdot)\rangle^2}{\|\bar{u}_+^\prime+v_y(x,\cdot)\|^4}
      =2I_1+2I_2.
    \end{split}
  \]
  
  Since $u$ is a solution of \eqref{elliptic} and $\bar{u}_+$ solves
  (\ref{system}) we have
  \[
  u_{xx}(x,\cdot)=W_u(u(x,\cdot))-W_u(\bar{u}_+(\cdot-h(x)))-\Big(u(x,\cdot)-\bar{u}_+(\cdot-h(x))\Big)_{yy}.
  \]
  Then, recalling the definition of the operator $T$ and that $v(x,\cdot)=u(x,\cdot+h(x))-\bar{u}_+$,
  we obtain
  \begin{equation}
    \begin{split}
      &I_1=\langle W_u(\bar{u}_++v(x,\cdot))-W_u(\bar{u}_+)-v_{yy}(x,\cdot),v(x,\cdot)\rangle\\
      &=\langle W_u(\bar{u}_++v(x,\cdot))-W_u(\bar{u}_+)-W_{uu}(\bar{u}_+)v(x,\cdot),v(x,\cdot)\rangle
      +\langle Tv(x,\cdot),v(x,\cdot)\rangle.
    \end{split}
    \label{I1}
  \end{equation}
  Now we observe that a standard computation yields 
  \[
  J_\R(u(x,\cdot))-c_0=\frac{1}{2}\langle Tv(x,\cdot),v(x,\cdot)\rangle+\int_\R f_W(x,y) dy,
  \]
  where 
  \[
  f_W=W(\bar{u}_++v)-W(\bar{u}_+)-W_u(\bar{u}_+)v
  -\frac{1}{2}W_{uu}(\bar{u}_+)v\cdot v.
  \]
  From \eqref{vy}, $q(x)=\|v(x,\cdot)\|\leq p$ and \eqref{linfty} it
  follows, with $C_W>0$ a suitable constant,
  \[
  \vert f_W(x,y)\vert\leq C_W\vert v(x,y)\vert^3\leq
  C_1C_W\|v(x,\cdot)\|^\frac{2}{3}\vert v(x,y)\vert^2,
  \]
  and therefore
  
  \begin{equation}
    \langle Tv(x,\cdot),v(x,\cdot)\rangle\geq 2(J_\R(u(x,\cdot))-c_0)-C\|v(x,\cdot)\|^\frac{8}{3}.
    \label{Tvv-W}
  \end{equation} 
  Introducing this estimate into \eqref{I1} and observing that the
  other term in the right hand side of \eqref{I1} can also be
  estimated by a constant times $\|v(x,\cdot)\|^\frac{8}{3}$ we finally
  obtain
  \[
  I_1\geq  2(J_\R(u(x,\cdot))-c_0)-C\|v(x,\cdot)\|^\frac{8}{3}.
  \]
  To estimate $I_2$ we note that from \eqref{vy} and \eqref{xenergy},
  provided $q^* = 2p$ is sufficiently small, we get
  \[
  \|v_x(x,\cdot)\|^2\leq 2\|u_x(x,\cdot)\|^2\leq
  4(J_\R(u(x,\cdot))-c_0),\;\;x\in\Bigl[x_p,\frac{L}{2}-x_p\Bigr],
  \] 
  where we have also used \eqref{ux-hamilton}. This 
  and \eqref{vy} imply
  \[
  \vert I_2\vert\leq Cp(J_\R(u(x,\cdot))-c_0),
  \]
  for some constant $C>0$ and we obtain
  \[
  (I_1+I_2)\geq(2-Cp)(J_\R(u(x,\cdot))-c_0)
  \geq\frac{1}{2}\mu\|v(x,\cdot)\|^2,\;\;x\in\Bigl[x_p,\frac{L}{2}-x_p\Bigr]
  \]
  and \eqref{ElIn0} is established.
  
  From \eqref{ElIn0} and the comparison principle we have
  \begin{equation}
    \|v(x,\cdot)\|^2\leq\varphi(x),\;\;x\in\Bigl[x_p,\frac{L}{2}-x_p\Bigr]
    \label{v<phi}
  \end{equation}
  where 
  \[
  \varphi(x)=p^2\frac{\cosh{\sqrt{\mu}(x-\frac{L}{4})}}
         {\cosh{\sqrt{\mu}(x_p-\frac{L}{4})}}
  \]
  is the solution of the problem
  \[
  \left\{\begin{array}{l}
  \varphi^{\prime\prime}=\mu\varphi,\;\;x\in[x_p,\frac{L}{2}-x_p],\\\\
  \varphi(x_p)=\varphi(\frac{L}{2}-x_p)=p^2.
  \end{array}\right.
  \]
  Then \eqref{inx>xp} follows from \eqref{v<phi} and 
  \[
  \varphi(x)\leq
  2p^2e^{-\sqrt{\mu}(x-x_p)},\;\;x\in\Bigl[x_p,\frac{L}{4}\Bigr].
  \]
  This concludes the proof.
\end{proof}

To finish the proof of Theorem \ref{periodic1} it remains to show that
there is a sequence $u^{L_j}$, $L_j\rightarrow+\infty$, that converges
to a heteroclinic connection between suitable translates of
$\bar{u}_\pm$. Indeed, once this is established, a suitable
translation $\eta$ in the $y$ direction yields the sequence
$u^{L_j}(x,y-\eta)$ and the heteroclinic $u^H$ in Theorem
\ref{periodic1}. From (\ref{C2alpha}) it follows that there exists a
subsequence, still denoted by ${L_j}$, and a classical solution
$u^\infty:\R^2\rightarrow\R^m$ of (\ref{elliptic}) such that we have
\begin{equation}
  \lim_{j\rightarrow+\infty}u^{L_j}(x,y)=u^\infty(x,y),
  \label{limLinfty}
\end{equation}
in the $C^2$ sense in compacts. Moreover $u^\infty$ satisfies the
exponential estimates \eqref{Exp-ind-L} and (\ref{DExp-ind-L}). This
implies that the convergence in (\ref{limLinfty}) is in the $C^2$
sense in any set of the form $[-\lambda,\lambda]\times\R$.  Set
$u_j=u^{L_j}$ and denote by $h_j$ and $v_j$ the functions determined
by the decomposition (\ref{decompositionL}) of $u_j$:
\begin{equation}
  \begin{split}
    &u_j(x,y)=\bar{u}_+(y-h_j(x))+v_j(x,y-h_j(x)),\\
    &\langle v_j,\bar{u}_+^\prime\rangle=0.
  \end{split}
  \label{decomLj}
\end{equation}

On the basis of Remark \ref{q-suff}, $v_j$ and its first and second
derivatives satisfy (\ref{vDExp-ind-L}). Therefore \eqref{vy} shows
that, under the standing assumption of $q^*>0$ small, we can control
the size of $\|(v_j)_y(x,\cdot)\|$ and, proceeding as in the
derivation of \eqref{vx-in-sp}, we obtain from \eqref{hprime}
\[
\vert h_j^\prime(x)\vert\leq C\|(v_j)_x(x,\cdot)\|\leq
C(J_\R(u_j(x,\cdot))-c_0)^\frac{1}{2},\;\;x\in\Bigl[l_{p},\frac{L_j}{4}\Bigr].
\]
 On the other hand from \eqref{Tvv-W} and \eqref{vy} we get 
 \[
 J_\R(u(x,\cdot))-c_0\leq C(\|v_y(x,\cdot)\|^2+\|v(x,\cdot)\|^2+\|v(x,\cdot)\|^\frac{8}{3})\leq
 C\|v(x,\cdot)\|,
 \]
 and we conclude
 \begin{equation}
 \vert h_j^\prime(x)\vert\leq C\|v_j(x,\cdot)\|^\frac{1}{2}\leq Ce^{-\frac{1}{4}\sqrt{\mu}(x-l_{p})},\;\;x\in\Bigl[l_{p},\frac{L_j}{4}\Bigr] 
 \label{hprime-in}
 \end{equation}
 where we have also used \eqref{inx>xp}.

This and the fact that, as we have seen in Remark \ref{q-suff},
$h_j(x)$ is bounded independently of $j$, imply that by passing to a
subsequence if necessary, we can assume that there is a Lipschitz
continuous and bounded map $h^\infty:[l_{p},+\infty)\rightarrow\R$
  such that
  \[
  \lim_{j\rightarrow+\infty}h_j(x)=h^\infty(x),\;\;x\in[l_{p},+\infty), 
  \]
uniformly in compacts. It follows that we can pass to the limit in
(\ref{decomLj}) and obtain in particular that there exists the limit 
$v^\infty:[l_{p},+\infty)\times\R\rightarrow\R^m$ of
  \[
  \lim_{j\rightarrow+\infty}v_j(x,y)=v^\infty(x,y),
  \]
and the convergence is in $L^2$ and in $L^\infty$ in sets of the form 
$[l_{p},l]\times\R$. The functions $h^\infty$ and $v^\infty$
coincide with the functions determined by the decomposition
(\ref{decompositionL}) of $u^\infty$. Moreover from (\ref{inx>xp}) and
\eqref{hprime-in} we have that $q^\infty(x)=\|v^\infty(x,\cdot)\|$ and
$h^\infty$ satisfy
\[
\begin{split}
&q^\infty(x)\leq C^*e^{-c^*x},\;\;x\geq 0,\\
&h^\infty(x)\leq Ce^{-\frac{1}{4}\sqrt{{\mu}}(x-l_{p})},\;\;x\geq l_{p}. 
\end{split}
\]
The first of these estimates shows that , for $x\rightarrow+\infty$,
$u^\infty(x,\cdot)$ converges in the $L^2$ sense to the manifold of
the translates of $\bar{u}_+$. The estimate for $h^\infty$ shows that
there exists $\eta=\lim_{x\rightarrow+\infty}h^\infty(x)$ and
therefore that actually, for $x\rightarrow\infty$, $u^\infty(x,\cdot)$
converges, to a specific element of that manifold. This, taking also
into account the symmetry properties of $u^\infty$ implies that indeed
$u^\infty$ is a heteroclinic solution of \eqref{elliptic} that
connects translates of $\bar{u}_\pm$.

This concludes the proof of Theorem \ref{periodic1}.

\appendix
\section{Appendix}

We present an elementary proof of Lemma~\ref{away}, that we restate as
Proposition\ref{W-distance}.

\begin{proposition}
  Assume that $W:\R^m\rightarrow\R$ is of class $C^3$, $a_\pm$ are non
  degenerate, and $u\in\mathcal{H}^1=\bar{\mathrm{u}}+H^1(\R;\R^m)$.
  
 Then, for each $p>0$ there is $e_p>0$ such that
\item
  \begin{equation}
    \|u-\bar{u}_\pm(\cdot-r)\|_1\geq q_1^u\geq {p},\;\;r\in\R.
    \label{W-distance}
  \end{equation}
  implies
  \[
  J_\R(u)-c_0\geq e_{p}.
  \]
  Moreover $e_{p}$ is continuous in $p$ and for $p\leq\|v\|_1$ small it results
  \[
    e_{p}\leq J_\R(\bar{u}+v)-c_0\leq C^1\|v\|_1^2,\;\;v\in H^1(\R;\R^m),\;\bar{u}\in\{\bar{u}_-,\bar{u}_+\},
  \]
  with $C^1>0$ a constant.
  \label{away1}
\end{proposition}
\begin{proof}
If $u$ satisfies (\ref{W-distance}) and has $J_\R(u)\geq 2c_0$ we can take $e_p=c_0$. It follows that in the proof we can assume
\begin{equation}
J_\R(u)<2c_0.
\label{E-bounded}
\end{equation}
Note also that $u\in\mathcal{H}^1$ implies
\begin{equation}
  \lim_{s\rightarrow\pm\infty}u(s)=a_\pm.
  \label{ConVerge}
\end{equation}
and  set
\begin{equation}
  q_0=\min\Bigl\{r_0,\frac{\gamma^2}{8C_W}\Bigr\},\quad 
  \label{q}
\end{equation}
where 
 $r_0$ and $\gamma$ are the constants in (\ref{second-derW})
  and 
 \[C_W=\max\{\vert W_{uuu}(a_\pm+z)\vert: \vert z\vert\leq3r_0\}.\]
Given $q\in(0,q_0)$ define
\[
\begin{split}
  &J_z^+(q)=\min_{v\in\mathcal{V}_z^+(q)} J(v),\\
  &\mathcal{V}_z^+(q)=\{v\in H^1_{loc}((0,\tau^v);\R^m): v(0)=z, \vert z-a_+\vert=q,\;\lim_{s\rightarrow\tau^v}v(s)=a_+\},\\
  &J^-(q)=\min_{v\in\mathcal{V}^-(q)} J(v),\\
  &\mathcal{V}^-(q)=\{v\in H^1_{loc}((0,\tau^v);\R^m):\vert v(0)-a_+\vert=q,\;\lim_{s\rightarrow\tau^v}v(s)=a_-\},\\
  &J_0(q)=\min_{v\in\mathcal{V}_0(q)} J(v),\\
  &\mathcal{V}_0(q)=\{v\in H^1((0,\tau^v);\R^m):\vert v(0)-a_+\vert=q_0,\;\vert v(\tau^v)-a_+\vert=q\}.
\end{split}
\]
Observe that there exists a positive functions $\psi:(0,q_0)\rightarrow\R$ that converges to zero with $q$ and
 satisfies
\[
 J_z^+(q)\leq\psi(q).
\]
Note also that $J_\R(\bar{u}_\pm)=c_0$ and the minimality of $\bar{u}_\pm$ imply $J^-(q)+\psi(q)\geq c_0$ and therefore we have
\begin{equation}
c_0-\psi(q)\leq J^-(q).
\label{J-bound}
\end{equation}
For $u\in\mathcal{H}^1$ define
\[
\begin{split}
  &s^{u,-}(\rho)=\max\{s:\vert u(t)-a_-\vert\leq\rho,\;\text{ for }\;t\leq s\},\\
  &s^{u,+}(\rho)=\min\{s:\vert u(t)-a_+\vert\leq\rho,\;\text{ for }\;t\geq s\}.
\end{split}
\]
Since $\psi(q)\rightarrow 0$ as $q\rightarrow 0$ while
$\lim_{q\rightarrow 0}J_0(q)=\mathsf{J}_0$, $\mathsf{J}_0$ a positive
constant, we can fix $q=q(q_0)$ in such a way that
\begin{equation}
2J_0(q(q_0))-\psi(q(q_0))\geq \mathsf{J}_0.
\label{qdef}
\end{equation}
We claim that in this proposition it suffices to consider only maps
that satisfy the condition
\begin{equation}
s^{u,+}(q_0)-s^{u,-}(q_0)\leq\frac{2c_0}{W_m(q(q_0))},
\label{s+-bounds}
\end{equation}
where $W_m(t)=\min_{a\in\{a_-,a_+\},\vert z\vert\geq t}W(a+z)$.
To see this set
\[
\begin{split}
&\bar{s}^{u,-}=\max\{s:\vert u(s)-a_-\vert=q(q_0)\},\\
&\bar{s}^{u,+}=\min\{s:\vert u(s)-a_+\vert=q(q_0)\},
\end{split}
\]
and observe that the definition of $\bar{s}^{u,\pm}$ implies $\vert u(s)-a_\pm\vert>q(q_0)$, for $s\in(\bar{s}^{u,-},\bar{s}^{u,+})$. It follows 
\begin{equation}
  (\bar{s}^{u,+}-\bar{s}^{u,-})W_m(q(q_0))\leq 2c_0.
  \label{bars-bound}
\end{equation}
Assume first that
\begin{equation}
  \begin{split}
    &\vert u(s)-a_-\vert<q_0,\;\text{ for }\;s\in(-\infty,\bar{s}^{u,-}),\\
    &\vert u(s)-a_+\vert<q_0,\;\text{ for }\;s\in(\bar{s}^{u,+},+\infty).
  \end{split}
  \label{inside}
\end{equation}
In this case we have
\[
\bar{s}^{u,-}<s^{u,-}(q_0)<s^{u,+}(q_0)<\bar{s}^{u,+},
\]
that together with \eqref{bars-bound} implies \eqref{s+-bounds}.  Now
assume that \eqref{inside} does not hold and there exists
$s^*\in(\bar{s}^{u,+},+\infty)$ such that $\vert u(s^*)-a_+\vert=q_0$
(or $s^*\in(-\infty,\bar{s}^{u,-})$ such that $\vert
u(s^*)-a_-\vert=q_0$). For definiteness we consider the first
eventuality, the other possibility is discussed in a similar way. To
estimate the energy of $u$ we focus on the intervals
$(-\infty,\bar{s}^{u,+})$, $(\bar{s}^{u,+},s^{u,+}(q(q_0)))$, and
$(s^{u,+}(q(q_0)),+\infty)$. We have
$J_{(-\infty,\bar{s}^{u,+})}(u)\geq J^-(q(q_0))$ and since
$s^*\in(\bar{s}^{u,+},s^{u,+}(q(q_0)))$ we also have
$J_{(\bar{s}^{u,+},s^{u,+}(q(q_0)))}(u)\geq 2J_0(q(q_0))$. This,
\eqref{J-bound} and \eqref{qdef} imply
\[
\begin{split}
&J_\R(u)\geq J_{(-\infty,\bar{s}^{u,+})}(u)+ J_{(\bar{s}^{u,+},s^{u,+}(q(q_0)))}(u)
\geq J^-(q(q_0))+2J_0(q(q_0))\\
&\geq c_0-\psi(q(q_0))+2J_0(q(q_0))\geq c_0+\mathsf{J}_0.
\end{split}
\]
This completes the proof of the claim. Indeed this computation shows
that, if $s^*$ with the above properties exists, then we can take
$e_p=\mathsf{J}_0$.

Since $J_\R$ is translation invariant we can also restrict ourselves
to the set of the maps that satisfy
\begin{equation}
  -s^{u,-}(q(q_0))=s^{u,+}(q(q_0))\leq\frac{c_0}{W_m(q(q_0))}.
  \label{+=-}
\end{equation}
and assume that also $\bar{u}_\pm$ satisfy \eqref{+=-}.  We remark
that the set of maps that satisfy \eqref{E-bounded} and
\eqref{s+-bounds} is equibounded and equicontinuous. Indeed
\eqref{E-bounded} implies
\[
  \vert u(s_1)-u(s_2)\vert\leq\sqrt{2c_0}\vert s_1-s_2\vert^\frac{1}{2},
\]
which together with \eqref{s+-bounds} yield
\[
  \vert u(s)\vert\leq M_0:=\vert a_-\vert +3q_0+\sqrt{2c_0}\Bigl(\frac{2c_0}{W_m(q(q_0))}\Bigr)^\frac{1}{2}.
\]

We first prove the proposition with (\ref{W-distance}) replaced by
\item
  \begin{equation}
    \|u-\bar{u}_\pm(\cdot-r)\|\geq q^u\geq {p},\;\;r\in\R.
    \label{l2-distance}
  \end{equation}
  
Assume the proposition is false. Then there is a sequence
$\{u_j\}\subset\mathcal{H}^1$ that satisfies (\ref{ConVerge}) and
\item
  \[
\begin{split}
&\lim_{j\rightarrow+\infty}J_\R(u_j)=c_0,\\
&\|u_j-\bar{u}_\pm(\cdot-r)\|\geq {p},\;\;r\in\R.
\end{split}
\]
Since the sequence $\{u_j\}$ is equibounded and equicontinuous there
is a subsequence, still labeled $\{u_j\}$ and a continuous map
$\bar{u}:\R\rightarrow\R^m$ such that
\[
\lim_{j\rightarrow+\infty}u_j(s)=\bar{u}(s),
\]
uniformly in compact sets. From $\int_\R\vert u_j^\prime\vert^2<4c_0$
and the fact that $u_j$ is uniformly bounded, by passing to a further
subsequence if necessary, we have that $u_j$ converges to $\bar{u}$
weakly in $H_{\mathrm{loc}}^1(\R;\R^m)$. A standard argument then
shows that
\[
J_\R(\bar{u})=c_0,
\]
and therefore, by the assumption that $\bar{u}_\pm$ and their
translates are the only minimizers of $J_\R$, we conclude that
$\bar{u}$ coincides either with $\bar{u}_-(\cdot-r)$ or with
$\bar{u}_+(\cdot-r)$ with $\vert r\vert\leq\lambda_0$ where
$\lambda_0$ is determined by the condition that $\bar{u}$ satisfies
\eqref{+=-}.
\vskip.2cm Since $\lambda_0$ is fixed, from \eqref{bar-exp} it follows
that we can assume $|\bar{u(s)} - a_+|\leq Ke^{-ks}$ for $s>0$.  Fix a
number $l>\lambda_0$ such that
 \begin{equation}
 Ke^{-kl}\leq q_0,\quad\text{and}\quad \frac{K}{C_W}e^{-kl}\leq \frac{p^2}{8},
 \label{l}
 \end{equation}
 and observe that $\bar{u}$ restricted to the interval $[-l,l]$ is a
 minimizer of $J_{(-l,l)}(u)$ in the class of $u$ that satisfy $u(\pm
 l)=\bar{u}(\pm l)$. From this observation it follows
\begin{equation}
J_{(-l,l)}(u_j)\geq J_{(-l,l)}(\bar{u})-Cl\delta_j,
\label{in-l+l}
\end{equation}
where $C>0$ is a constant and $\delta_j=\max_\pm\vert u_j(\pm
l)-\bar{u}(\pm l)\vert$.

From the properties of $u$ and (\ref{bar-exp}) we have 
\begin{equation}\label{vj-small}
\vert u_j(s)-\bar{u}(s)\vert\leq\vert u_j(s)-a_+\vert+\vert\bar{u}(s)-a_+\vert\leq q_0+Ke^{-kl}\leq 2q_0,
\;\,\text{ for }\;s\geq l.
\end{equation}
We estimate the differences $J_{(-\infty,-l)}(u_j)-J_{(-\infty,-l)}(\bar{u})$ and $J_{(l,+\infty)}(u_j)-J_{(l,+\infty)}(\bar{u})$.
We have with $u_j=\bar{u}+v_j$
\begin{equation}
\begin{split}
& J_{(l,+\infty)}(u_j)-J_{(l,+\infty)}(\bar{u})
=\int_l^{+\infty}\Big(\bar{u}^\prime\cdot v_j^\prime+\frac{1}{2}\vert v_j^\prime\vert^2+W(\bar{u}+v_j)-W(\bar{u})\Big)ds\\
&=-\bar{u}^\prime(l)\cdot v_j(l)+\int_l^{+\infty}\Big(-\bar{u}^{\prime\prime}\cdot v_j+\frac{1}{2}\vert v_j^\prime\vert^2+W(\bar{u}+v_j)-W(\bar{u})\Big)ds\\
&=-\bar{u}^\prime(l)\cdot v_j(l)+\int_l^{+\infty}\Big(\frac{1}{2}\vert v_j^\prime\vert^2+W(\bar{u}+v_j)-W(\bar{u})-W_u(\bar{u})\cdot v_j\Big)ds\\
&\geq-2q_0Ke^{-kl}
+\int_l^{+\infty}\Big(\frac{1}{2}(\vert v_j^\prime\vert^2+W_{uu}(\bar{u})v_j\cdot v_j)\\
&\quad\quad+W(\bar{u}+v_j)-W(\bar{u})-W_u(\bar{u})\cdot v_j-\frac{1}{2}W_{uu}(\bar{u})v_j\cdot v_j\Big)ds
\end{split}
\label{J-J}
\end{equation}
Set $I(v_j)=W(\bar{u}+v_j)-W(\bar{u})-W_u(\bar{u})\cdot v_j-\frac{1}{2}W_{uu}(\bar{u})v_j\cdot v_j$. Then we have
\[\begin{split}
&I(v_j)=\int_0^1\int_0^1\int_0^1\rho^2\sigma W_{uuu}(\bar{u}+\rho\sigma\tau v_j)(v_j,v_j,v_j)d\tau d\sigma d\rho\\
&=\int_0^1\int_0^1\int_0^1\rho^2\sigma W_{uuu}(a_++(\bar{u}-a_+)+\rho\sigma\tau v_j)(v_j,v_j,v_j)d\tau d\sigma d\rho.
\end{split}\]
It follows $\vert I(v_j)\vert\leq 2q_0C_W\vert v_j\vert^2$. This and \eqref{J-J} imply 
\[
\begin{split}
& J_{(l,+\infty)}(u_j)-J_{(l,+\infty)}(\bar{u})\\
&\geq-2q_0Ke^{-kl}+\int_l^\infty\frac{1}{2}(\vert v_j^\prime\vert^2+\gamma^2\vert v_j\vert^2)ds-2q_0C_W\int_l^\infty\vert v_j\vert^2 ds\\
&\geq-\frac{\gamma^2}{4C_W}Ke^{-kl}+\frac{1}{4}\gamma^2\int_l^\infty\vert v_j\vert^2ds\\
&\geq-\gamma^2\frac{p^2}{32}+\frac{1}{4}\gamma^2\int_l^\infty\vert v_j\vert^2ds,
\end{split}
\]
where we have used (\ref{q}) and \eqref{l}.
From this, the analogous estimate valid in the interval $(-\infty,-l)$, and (\ref{in-l+l}) we obtain
\begin{equation}
\begin{split}
&0=\lim_{j\rightarrow+\infty}(J_\R(\bar{u}+v_j)-c_0)\\
&\geq\lim_{j\rightarrow+\infty}\Big(-Cl\delta_j-\gamma^2\frac{p^2}{16}
+\frac{1}{4}\gamma^2\Big(\int_{-\infty}^{-l}\vert v_j\vert^2ds+\int_l^\infty\vert v_j\vert^2ds\Big)\Big).
\end{split}
\label{lim}
\end{equation}
Since $v_j$ converges to $0$ uniformly in $[-l,l]$, for $j$ large we have
\[\int_{-l}^l\vert v_j\vert^2\leq\frac{p^2}{2}\]
and therefore from \eqref{l2-distance}
\[\int_{-\infty}^{-l}\vert v_j\vert^2ds+\int_l^\infty\vert v_j\vert^2ds\geq\frac{p^2}{2}.\]
This and  $\eqref{lim}$ imply
\[
\begin{split}
&0=\lim_{j\rightarrow+\infty}(J_\R(\bar{u}+v_j)-c_0)\\
&\geq\lim_{j\rightarrow+\infty}\Big(-Cl\delta_j-\gamma^2\frac{p^2}{16}
+\gamma^2\frac{p^2}{8}\Big)=\gamma^2\frac{p^2}{16}.
\end{split}
\]

This contradiction concludes the proof of the proposition when
(\ref{W-distance}) is replaced by (\ref{l2-distance}).  To complete
the proof we note that it suffices to consider the case $p\leq
2(2+\sqrt{2})\sqrt{c_0}=:2p_0$. Indeed \eqref{E-bounded} implies
$\|u^\prime\|\leq2\sqrt{c_0}$ that together with
$\|\bar{u}_\pm^\prime\|\leq\sqrt{2c_0}$ yields
\begin{equation}
\|u^\prime-\bar{u}_\pm^\prime(\cdot-r)\|\leq p_0,\;\;r\in\R.
\label{vprime<}
\end{equation}
It follows that $p > 2p_0$ implies
$\|u-\bar{u}_\pm^\prime(\cdot-r)\| > p_0$ and the existence of $e_p$
follows from the first part of the proof.

Set
\[
C_W^0=\max\{\vert W_{uu}(\bar{u}_\pm(s)+z)\vert:s\in\R,\;\vert z\vert\leq 2p_0\},
\]
and define $\tilde{p}=\tilde{p}(p)$ by
\[
\tilde{p}(p)=\frac{p}{\sqrt{2(1+C_W^0)}}.
\]
We distinguish the following alternatives:

a) $\|u-\bar{u}_\pm(\cdot-r)\|\geq\tilde{p},\;\text{ for }\;r\in\R$,
\vskip.2cm
b) there exists $\bar{r}\in\R$ and $\bar{u}\in\{\bar{u}_-,\bar{u}_+\}$ such that
\begin{equation}
\|u-\bar{u}(\cdot-\bar{r})\|< \tilde{p}.
\label{l2-small}
\end{equation}
In case a) the proposition is true from the first part of the proof
with $e_p=e_{\tilde{p}}$.
\vskip.2cm
Case b). From  (\ref{W-distance}) and (\ref{l2-small}) it follows
\begin{equation}
\|u^\prime-\bar{u}^\prime(\cdot-\bar{r})\|^2>p^2-\tilde{p}^2.
\label{W1>}
\end{equation}
For simplicity we write $\bar{u}$ instead of $\bar{u}(\cdot-\bar{r})$
and set $v=u-\bar{u}$. Note that from \eqref{vprime<},
\eqref{l2-small} and $\tilde{p}\leq p_0$ it follows
\[
\vert v(s)\vert^2\leq2\int_{-\infty}^s\vert v(s)\vert\vert
v^\prime(s)\vert ds\leq2\|v\|\|v^\prime\|\leq4p_0^2.
\]

We compute

\begin{equation}
J(u)-c_0=\frac{1}{2}\|v^\prime\|^2+\int_\R\int_0^1\Big(W_u(\bar{u}+\tau
v)-W_u(\bar{u})\Big)vd\tau ds.
\label{jW11>}
\end{equation}
Since
\begin{equation}
\Bigl\vert \int_0^1\Big(W_u(\bar{u}+\tau
v)-W_u(\bar{u})\Big)vd\tau\Bigr\vert\leq\frac{1}{2}C_W^0\vert v\vert^2,
\label{>}
\end{equation}
 we have from (\ref{W1>}) and (\ref{jW11>})
\[
J(u)-c_0\geq\frac{1}{2}(p^2-\bar{p}^2)-\frac{1}{2}C_W^0\bar{p}^2=\frac{1}{4}p^2.
\]
This concludes the first part of the lemma. The last statement is a
consequence of the fact that $J_\R(u)$ is continuous in
$\mathcal{H}^1$ and of (\ref{jW11>}) and (\ref{>}).
\end{proof}

\bibliographystyle{plain}

\end{document}